\documentclass{birkjour}
\usepackage{graphics,graphicx,subfig,color}
\usepackage{amsmath,amsfonts,amscd,amssymb,bm,amsthm}
\usepackage{mathrsfs}
\usepackage{mathtools,eqparbox}
\usepackage{fancyhdr}
\usepackage{listings}
\usepackage{xcolor}
\usepackage{lingmacros}
\usepackage{blindtext}
\usepackage{tree-dvips}
\usepackage{empheq}
\usepackage{enumerate}
\usepackage{array}
\usepackage[framemethod=tikz]{mdframed}
\usepackage[section]{placeins}
\usepackage{marginnote}
\usepackage{float}
\usepackage{textcomp}
\usepackage[makeroom]{cancel}
\usepackage{verbatim}
\newtheorem{theorem}{Theorem}

\newtheorem{algorithm}{Algorithm}

\newtheorem{remark}{Remark}
\newtheorem{step}{\sc Step}

\numberwithin{equation}{section}

\newcommand{\eqmath}[3][l]{\eqmakebox[#2][#1]{$\displaystyle\if#1l{}\fi#3$}}
\newcommand{\ftext}[2]{\makebox[#1][l]{\textup{#2}}}

\usepackage{cite}
\numberwithin{equation}{section}
\usepackage{color}

\allowdisplaybreaks

\begin{document}

\title[Refactorization of Cauchy’s method applied to FSI]{Refactorization of Cauchy’s method: a second--order partitioned method for fluid--thick structure interaction problems}



\author[Bukac]{Martina Buka\v{c}}

\address{Department of Applied and Computational Mathematics and Statistics, University of Notre Dame, Notre Dame, IN 46556, USA.  }\email{mbukac@nd.edu} 

\author[Seboldt]{Anyastassia Seboldt}

\address{Department of Applied and Computational Mathematics and Statistics, University of Notre Dame, Notre Dame, IN 46556, USA.  }\email{aseboldt@nd.edu} 

\author[Trenchea]{Catalin Trenchea}

\address{Department of Mathematics, University of Pittsburgh, Pittsburgh, PA 15260,
USA.} 
\email{trenchea@pitt.edu}





\begin{abstract}
This work focuses on the derivation and the analysis of a novel, strongly-coupled partitioned method for fluid-structure interaction problems. 
The flow is assumed to be  viscous and incompressible, and the structure is modeled using 
linear elastodynamics equations. We assume that the structure is thick, i.e., modeled using the same number of spatial dimensions as fluid. Our newly developed numerical method is based on generalized Robin boundary conditions, 
as well as on the refactorization of the Cauchy's one-legged `$\theta$-like' method, written as a sequence of Backward Euler--Forward  Euler steps used to discretize the problem in time. 
This family of methods, parametrized by $\theta$, is B-stable for any $\theta \in [\frac12 , 1]$ and second-order accurate for $\theta = \frac12 + \mathcal{O}(\tau)$, where $\tau$ is the time step. In the proposed algorithm, the fluid and structure  sub-problems, discretized using the  Backward Euler scheme, are first solved iteratively
until convergence. Then, the variables are linearly extrapolated, 
equivalent to solving  Forward Euler problems. 
We prove that the  iterative procedure  is convergent, and that the proposed method is  stable provided $\theta \in [\frac12,1]$. 
Numerical examples, based on the finite element discretization in space,  explore convergence rates using different values of parameters in the problem, and compare our method to other strongly-coupled partitioned schemes from the literature. We also compare our  method to both a monolithic and a non-iterative partitioned solver on a benchmark problem with parameters within the physiological range of blood flow, obtaining an excellent agreement with the monolithic scheme.

\end{abstract}

\maketitle

\section{Introduction}
Fluid-structure interaction (FSI) describes a specific type of problem that involves the highly non-linear relationship between a fluid and deformable structure. The importance of solving FSI problems is made clear when one simply sits back to observe everyday life -- in the wind that blows across an airplane wing or bridge, a vessel or fish that ventures across the ocean, or even someone's heart sending a pulse of blood through an artery. Since these ubiquitous occurrences cannot be over-emphasized with their pertinent applications in the biomedical, engineering, and architectural realms,  FSI problems have received a lot of attention from both theoretical and computational perspectives. In particular,  with increasing medical demand revolving around hemodynamics-related problems, engineering advancements for the understanding and improvement of aeronautical and naval applications, and demand for more sustainable energy harvesting, there is a high demand for fast and accurate numerical solvers for FSI problems.

Two main methodologies for numerically solving FSI problems are monolithic  and partitioned schemes. Both monolithic and partitioned algorithms begin with the same set of governing equations describing the motion of the fluid and solid, as well as their interaction at the interface, but differ in the way they are solved. Monolithic schemes~\cite{bazilevs2008isogeometric,deparis2003acceleration,gerbeau2003quasi,nobile2001numerical,gee2011truly,ryzhakov2010monolithic,hron2006monolithic,bathe2004finite} solve the governing equations in one, fully-coupled, algebraic system, with implicitly imposed boundary conditions.  While in this way the fluid and structure remain strongly-coupled,  the large system may likely become ill-conditioned and require specially designed preconditioners. Furthermore, since  all of the unknown variables are simultaneously solved, this approach is rather computationally expensive~\cite{gee2011truly,badia2008modular,heil2008solvers}. On the other hand, partitioned methods~\cite{degroote2008stability,farhat2006provably,bukavc2012fluid,badia2009robin,BorSunMulti,Fernandez2012incremental,fernandez2013fully,nobile2008effective,hansbo2005nitsche,lukavcova2013kinematic,banks2014analysis,banks2014analysis2,oyekole2018second,bukavc2016stability} use separate solvers for the fluid and structure sub-problems while enforcing coupling at the interface using a variety of potential boundary conditions, e.g. Dirichlet, Neumann, and Robin. In this way, each sub-problem has fewer unknowns and is better conditioned, making the partitioned scheme less computationally expensive in comparison to the monolithic solver. However, stability issues often arise as a result of the coupling at the interface unless the design and implementation of a partitioned scheme is carefully developed.

Furthering the varying difficulties that arise amongst partitioned methods, certain physical factors play another role in making FSI problems especially challenging to solve. One such instance occurs in hemodynamics, where the structure and fluid densities are comparable, jeopardizing the stability due to \textit{the added mass effect}~\cite{causin2005added}. In the case of thin structures, enforcing the structure's mass in the fluid problem is successfully done using different Robin boundary conditions on the interface~\cite{oyekole2018second,bukavc2016stability,Fernandez2012incremental,lukavcova2013kinematic}. However, such approaches cannot be directly applied when the dimension of the structure is the same as that of the fluid, i.e., in case of thick structures. 

Whenever these {added mass effect} scenarios occur in FSI cases with thick structures, classical Dirichlet-Neumann approaches are notorious for falling short because they are proven to be unconditionally unstable~\cite{causin2005added}. Even when sub-iterations are implemented in these cases  in order to enforce stability (resulting in strongly-coupled partitioned schemes), convergence issues still arise. Hence, alternative options are to use the Robin-Dirichlet, Robin-Neumann, and Robin-Robin types of boundary conditions to be implemented on the fluid-structure interface~\cite{nobile2008effective,badia2009robin,badia2008fluid,gerardo2010analysis,degroote2011similarity,gigante2019stability}. These have been intensively analyzed for their efficacy in maintaining stability and convergence, and therefore widely used in different applications.
 We also mention the fictitious-pressure and fictitious-mass strongly-coupled algorithms proposed in~\cite{baek2012convergence,yu2013generalized}, in which  the added mass effect is accounted for by incorporating additional terms into governing equations. 
 
 When FSI problems with thick structures are solved using partitioned methods without sub-iterations, sub-optimal convergence in time may become an issue as seen in~\cite{burman2020fully,burman2009stabilization,fernandez2015generalized,bukavc2016stability,paper1}.  In particular, a  partitioned, loosely-coupled scheme  based on the Nitsche's penalty method was proposed in~\cite{burman2009stabilization,burman2013unfitted}, where some interface terms were time-lagged  in order to decouple the fluid and structure sub-problems. The  scheme is proved to be  stable under a CFL condition if a weakly consistent stabilization term that includes pressure variations at the interface is added. It was shown that the rate of convergence in time is $\mathcal{O}(\tau^{\frac12})$, which was then corrected to obtain $\mathcal{O}(\tau)$ by proposing a few defect-correction sub-iterations. 
A non-iterative, generalized Robin-Neumann partitioned scheme based on an interface operator accounting for the
solid inertial effects within the fluid, has been proposed in~\cite{fernandez2015generalized}. The scheme has been analyzed on a linear FSI problem and shown to be stable under a time-step condition. However, a time step $ \tau = \mathcal{O} (h^{\frac32})$, where $h$ is the mesh size, is needed to achieve a first-order accuracy.  An alternative class of Added-Mass Partitioned algorithms has also been developed in~\cite{banks2014analysis2,banks2014analysis,serino2019stable}.
A non-iterative, partitioned algorithm for FSI with thick structures was first proposed in~\cite{banks2014analysis2}.  It was shown that the algorithm is stable  under a condition on  the time step, which depends on the structure parameters. Although the authors do not derive the convergence rates, their numerical results indicate that the scheme is second-order accurate in time. In~\cite{serino2019stable}, the previously developed algorithms have been extended to finite deformations, and the explicit fluid solver was replaced by a fractional-step implicit-explicit scheme.

 In our previous work~\cite{bukavc2014modular}, using the operator splitting approach,  we developed a partitioned scheme for FSI with a thick, linearly viscoelastic structure. However, the assumption that the structure is viscoelastic was necessary in the derivation of the scheme, and the solid viscosity was solved implicitly with the fluid problem. Furthermore, the scheme was shown to be stable only under a condition on the time step~\cite{bukavc2016stability}. More recently, we proposed a  loosely-coupled method for FSI with thick structures~\cite{paper1} based on generalized Robin coupling conditions, similar to a parallel work  by Burman et al~\cite{burman2019stability}. We proved that the method is unconditionally stable on a moving domain FSI problem using energy estimates. However, the method is shown to be only $\mathcal{O}(\tau^{\frac12})$ in time~\cite{paper1,burman2020fully}.

 In this paper, we are interested in solving FSI problems with thick, elastic structures and with a particular interest in applications with similar fluid and solid densities. The time dependent Stokes equations are used to describe the incompressible, viscous fluid, and the linearly elastic equations are employed for the solid. We propose a novel partitioned, strongly-coupled scheme, where the interface conditions are enforced with the use of generalized Robin coupling conditions. These conditions are obtained by linearly combining  the kinematic (Dirichlet) and dynamic (Neumann) interface conditions, along with the use of a combination parameter, $\alpha$, whose purpose is to dictate the emphasis on either the kinematic or dynamic  condition.  The time discretization is based on the one-legged `$\theta$-like' method proposed by Cauchy~\cite{MR1013996}. This family of methods, parametrized by $\theta$,  is B-stable for any $\theta \in [\frac12 , 1]$ and second-order accurate for $\theta = \frac12 + \mathcal{O}({ \tau})$. We note that  $\theta=\frac12$ corresponds to the
 midpoint rule. 
In this work,  similarly as  in~\cite{burkardt2020refactorization}, we refactorize the Cauchy's method by writing it as sequential Backward Euler (BE)--Forward Euler (FE) problems.
This results in first solving the partitioned scheme on $[t^n, t^{n+\theta}]$ discretized using the Backward Euler method in which we sub-iterate the fluid and structure sub-problems until convergence. Then, the variables are linearly extrapolated, equivalent to solving the Forward Euler problems on $[t^{n+\theta}, t^{n+1}]$.
  We prove that the sub-iterative process is convergent and that the proposed method is stable provided $\theta \in [\frac12, 1]$. 
  
The promising theoretical results are further illustrated in numerical examples, where the finite element method is used to discretize the problem in space. The first example uses the method of manufactured solutions to investigate convergence rates across varying values of the combination parameter, $\alpha$, used in the derivation of the generalized Robin coupling condition, a parameter $\theta$, used in the time discretization, and the tolerance, $\epsilon$, used to control the sub-iterative procedure. The examples successfully meet, and in some cases exceed, the expectations of the second-order convergence in time. We also compare the average number of sub-iterations across multiple sub-iterative methods to our novel scheme in order to illustrate the reduced computational cost of our iterative approach. In particular, we show that the proposed method requires significantly fewer sub-iterations than the aforementioned Robin-Neumann and Robin-Robin methods.
In the second numerical example, we model the flow in  a two-dimensional channel with parameters similar to that of blood flow in order to show our method's competitiveness when compared to both a monolithic and a loosely-coupled, partitioned method.


The outline of this paper is as follows: We define the problem in Section 2 and elaborate upon our novel numerical method in Section 3. Convergence of the iterative procedure is analyzed and proven in Section 4, following up with the stability analysis in Section 5. Numerical examples are presented in Section 6. Section 7 highlights the conclusions of the main topics and results presented in this paper.

\section{Problem description}

We consider the interaction between an incompressible, viscous fluid and a linearly elastic structure. The fluid domain is denoted by $\Omega_F$ and the structure domain by $\Omega_S$. We assume that $\Omega_F, \Omega_S \subset \mathbb{R}^d, d=2,3$ are open, smooth sets of the same dimension, and that the fluid-structure interface $\Gamma$  is the common boundary between the two domains, i.e. $ \partial {\Omega}_F \cap \partial {\Omega}_S = \varnothing,  \partial \bar{\Omega}_F \cap \partial \bar{\Omega}_S=\Gamma, $ (see Figure~\ref{domain}). The fluid inlet and outlet boundaries are designated by $\Gamma_F^{in}$ and $\Gamma_F^{out},$ respectively, the  solid inlet and outlet boundaries by $\Gamma_S^{in}$ and $\Gamma_S^{out},$ respectively, and the external solid boundary by $\Gamma_S^{ext}$. Therefore,  $\partial \Omega_F = \Gamma_F^{in} \cup \Gamma_F^{out} \cup \Gamma$ and $\partial \Omega_S = \Gamma_S^{in} \cup \Gamma_S^{out} \cup \Gamma_S^{ext} \cup \Gamma$. 
In the following, similarly as in~\cite{burman2019stability,burman2020fully,bukavc2016stability,burman2009stabilization,fernandez2015generalized}, we assume that the structure deformation is infinitesimal,  that fluid-structure interaction is linear, and that the fluid domain does not change. 
\begin{figure}[ht]
        \centering{
\includegraphics[scale=0.7]{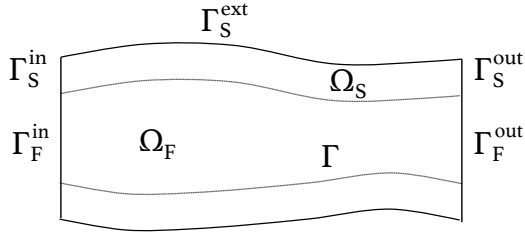}
        }
        \caption{ Fluid domain $\Omega_F$ and  structure domain $\Omega_S$, separated by a common interface $\Gamma$.}
         \label{domain}
 \end{figure}

To model the fluid flow, we use the time dependent Stokes equations, given as follows:
\begin{subequations}
\begin{align}\label{NSale1}
& \rho_F   \partial_t \boldsymbol{u}  = \nabla \cdot \boldsymbol\sigma_F(\boldsymbol u, p) + \boldsymbol f_F&  \textrm{in}\; \Omega_F\times(0,T), \\
 \label{NSale2}
&\nabla \cdot \boldsymbol{u} = 0 & \textrm{in}\; \Omega_F \times(0,T),
\end{align}
\label{flow}
\end{subequations}
where $\boldsymbol{u}$ is the fluid velocity,  $\rho_F$ is fluid density,  $\boldsymbol\sigma_F$ is the fluid stress tensor and $\boldsymbol f_F$ is the forcing term. 
For a Newtonian fluid, the stress tensor is given by $\boldsymbol\sigma_F(\boldsymbol u,p) = -p \boldsymbol{I} + 2 \mu_F \boldsymbol{D}(\boldsymbol{u}),$ where  $p$ is the fluid pressure,
$\mu_F$ is the fluid viscosity and  $\boldsymbol{D}(\boldsymbol{u}) = (\nabla \boldsymbol{u}+(\nabla \boldsymbol{u})^{T})/2$ is the strain rate tensor.
At the inlet and outlet sections we prescribe Neumann boundary conditions:
\begin{subequations}
\begin{align}
& \boldsymbol {\sigma}_F \boldsymbol{ n}_{F} = -p_{in} (t) \boldsymbol{n}_{F}  &  \textrm{on} \; \Gamma_{F}^{in} \times (0,T), \label{inlet}  \\
& \boldsymbol {\sigma}_F \boldsymbol{ n}_{F} = -p_{out}(t)  \boldsymbol{n}_{F}  &  \textrm{on} \; \Gamma_{F}^{out} \times (0,T), \label{outlet} 
\end{align}
\label{neumann}
\end{subequations}
where $\boldsymbol n_F$ is the outward unit normal to the  fluid domain.


To model the elastic structure, we use the elastodynamics equations written in the first order form as
\begin{subequations}
\begin{align}
&{\rho}_S \partial_{t} {\boldsymbol \xi}  =  {\nabla} \cdot \boldsymbol \sigma_S(\boldsymbol \eta) &  \textrm{in}\; {\Omega}_S\times(0,T), 
\label{solid}
\\
& \partial_{t} {\boldsymbol \eta}  =  \boldsymbol \xi &  \textrm{in}\; {\Omega}_S\times(0,T), 
\end{align}
\label{solid}
\end{subequations}
where $\boldsymbol{\eta}$ is the structure displacement,   $\boldsymbol{\xi}$ is the structure velocity, ${\rho}_S$ is the structure density and ${\boldsymbol \sigma_S}$ is the solid Cauchy stress tensor. To relate the stress and strain fields, we use the Saint-Venant Kirchhoff constitutive law given as
\begin{align*}
 \boldsymbol \sigma_S(\boldsymbol \eta) = 2 \mu_S \boldsymbol D(\boldsymbol \eta) + \lambda_S (\nabla \cdot \boldsymbol \eta) \boldsymbol I,
\end{align*}
where $\mu_S$ and $\lambda_S$ are Lam\'e constants.
We define a norm associated with the structure elastic energy as
\begin{align}
\| \boldsymbol \eta \|^2_S = 2 \mu_S  \| \boldsymbol D(\boldsymbol \eta)\|^2_{L^2(\Omega_S)} + \lambda_S \|\nabla \cdot \boldsymbol \eta\|^2_{L^2(\Omega_S)}.
\label{sNorm}
\end{align}
The structure is assumed to be fixed at the inlet and outlet boundaries:
\begin{equation}\label{homostructure1}
 {\boldsymbol \eta} = 0 \quad \textrm{on} \;\; {\Gamma}_S^{in} \cup {\Gamma}_S^{out} \times(0,T),
\end{equation}
and at the external structure boundary, ${\Gamma}_S^{ext}$, we impose:
\begin{equation}\label{homostructure2}
{\boldsymbol \sigma_S}  {\boldsymbol{ n}}_S =  0 \quad \textrm{on}  \;\; {\Gamma}_S^{ext} \times (0,T),
 \end{equation}
 where ${\boldsymbol n}_S$ is the outward normal to the  structure domain.

To couple the fluid and structure sub-problems, we prescribe the kinematic and dynamic coupling conditions~\cite{langer2018numerical,multilayered}, given as follows:

\noindent \emph{Kinematic (no-slip) coupling condition} describes the continuity of  velocity at the fluid-structure interface: 
\begin{equation}
{\boldsymbol{u}} = {\boldsymbol \xi}  \;\; \; \textrm{on} \; {\Gamma} \times (0,T), \label{kinematic}
\end{equation}  

\noindent \emph{Dynamic coupling condition} describes the continuity of stresses at the fluid-structure interface:
\begin{equation}
\boldsymbol \sigma_F \boldsymbol n_F +  {\boldsymbol \sigma_S } {\boldsymbol n}_S=0 \;\;\; \textrm{on} \; \Gamma \times (0,T). \label{dynamic}
 \end{equation}
 Initially, the fluid and  structure are assumed to be at rest:
\begin{equation}\label{initial}
\boldsymbol{u}=0 \quad  \textrm{in} \; {\Omega}_F,  \quad {\boldsymbol \eta}=0,
 \; {\boldsymbol \xi} =0 \quad \textrm{in} \; {\Omega}_S \quad {\rm at}\; t=0.
\end{equation}

\section{Numerical method}

Let  $t^n = n \tau$ for $n=0, \ldots, N$, where $\tau$ denotes the time step, and $t^{n+\theta} = t^n + \theta \tau$, for any $\theta\in [0,1]$ and for all $n\geq 0$. Let $z^n$ denote the approximation of a time-dependent function $z$ at time level $t^n$.  
The proposed algorithm is based on the refactorization of the Cauchy's one-legged `$\theta$-like' method. In particular, for an initial value problem
$y'=f(t,y(t))$, the Cauchy's one-legged `$\theta$-like' method is given as
\begin{align*}
\frac{y^{n+1}-y^n}{\tau} = f(t^{n+\theta}, y^{n+\theta}),
\end{align*}
for $\theta \in [0,1]$, where $y^{n+\theta} = \theta y^{n+1} + (1-\theta) y^n$.
The problem above  can be solved in the BE-FE fashion~\cite{burkardt2020refactorization} as
\begin{align*}
& \mbox{BE:  }\quad  \frac{y^{n+\theta}-y^n}{ \theta \tau} = f(t^{n+\theta}, y^{n+\theta}),
\\
& \mbox{FE:  }\quad  \frac{y^{n+1}-y^{n+\theta}}{(1- \theta) \tau} = f(t^{n+\theta}, y^{n+\theta}).
\end{align*}
The  FE problem can also be written as a linear extrapolation given by
$$
y^{n+1} = \frac{1}{\theta} y^{n+\theta}-\left(\frac{1}{\theta}-1\right) y^n.
$$
 We note that the case when $\theta=\frac12$ corresponds to the midpoint rule. Using this approach, the main computational load of the algorithm is related to the computation of the BE steps, while computationally inexpensive linear extrapolations increase the accuracy of the scheme.

 While the governing equations will be discretized using this approach, we impose the generalized Robin coupling conditions at the interface. 
Similar as in~\cite{badia,badia2009robin}, we consider a linear combination of  coupling conditions~\eqref{kinematic}-\eqref{dynamic}
\begin{align*}
& \alpha \boldsymbol \xi+  \boldsymbol{\sigma}_S {\boldsymbol n}_S= \alpha \boldsymbol{u} - \boldsymbol \sigma_F  \boldsymbol n_F    \;\;\; \textrm{on} \; \Gamma \times (0,T), 
 \end{align*}
where $\alpha>0$ is a  combination parameter. As in~\cite{paper1}, by imposing~\eqref{dynamic} one more time, we introduce the following two  transmission conditions of Robin type:
\begin{align}
& \alpha \boldsymbol \xi +  \boldsymbol{\sigma}_S  {\boldsymbol n}_S= \alpha \boldsymbol{u}    - \boldsymbol \sigma_F \boldsymbol n_F \;\;\; \textrm{on} \; \Gamma \times (0,T), \label{cc1}
\\
  & 
  \alpha \boldsymbol \xi    - \boldsymbol \sigma_F   \boldsymbol n_F 
 = 
\alpha \boldsymbol{u} 
 - \boldsymbol \sigma_F   \boldsymbol n_F.
  \quad  \textrm{on} \; \Gamma \times (0,T).
   \label{cc2}
 \end{align}
These conditions will be used in the BE, sub-iterative part of our algorithm. In particular, by taking the values on the right-hand side  from the previous sub-iteration, equation~\eqref{cc1}  will serve as a Robin-type boundary condition for the BE structure sub-problem. Using the most recent values of the structure variables, equation~\eqref{cc2} will serve as a Robin-type boundary condition for the BE fluid sub-problem. 
The proposed partitioned numerical method is given in the following algorithm. 

\begin{algorithm}
\label{algorithm1}
Given ${\boldsymbol u}^{0}$ in $\Omega_F$, and $ {\boldsymbol \eta}^0,  {\boldsymbol \xi}^0$ in $\Omega_S$, 
we first need to compute $
p^{\theta}, p^{1+\theta}, {\boldsymbol u}^{1}, {\boldsymbol u}^{2}$ in $\Omega_F$, and $ 
{\boldsymbol \eta}^{1}, {\boldsymbol \eta}^2,  {\boldsymbol \xi}^{1},  {\boldsymbol \xi}^{2}
$ in $\Omega_S$ with a  second-order method. 
A monolithic method could be used.
Then, for all $n\geq 2$, compute the following steps:
\begin{step}
Set the initial guesses as the linearly extrapolated values:
\begin{align*}
& {\boldsymbol \eta}^{n+\theta}_{(0)}
=  \Big( 1 + \theta  \Big) {\boldsymbol \eta}^{n}  - \theta {\boldsymbol \eta}^{n-1} ,
\end{align*}
and similarly for ${\boldsymbol \xi}^{n+\theta}_{(0)}, {\boldsymbol u}^{n+\theta}_{(0)}$.
The pressure initial guess is defined as 
\begin{align*}
p^{n+\theta}_{(0)} = (1 + \tau) p^{n-1+\theta}
				-\tau  p^{n-2+\theta}
.
\end{align*}
For  $\kappa\geq 0$, compute until convergence the following \textbf{Backward Euler} partitioned problem:
\begin{subequations}
\begin{empheq}[left= \ftext{2.2em}{Solid:} \;  {\empheqlbrace \quad}]{alignat=2}
&
  \eqmath[l]{A}{ \frac{{\boldsymbol \eta}^{n+\theta}_{(\kappa+1)} - {\boldsymbol \eta}^{n}}{\theta \tau} 
= 
{\boldsymbol \xi}^{n+\theta}_{(\kappa+1)}} 
& \eqmath[r]{B}{\mbox{ in }\Omega_S,}
\\
&
  \eqmath[l]{A}{ \rho_S \frac{\boldsymbol{\xi}_{(\kappa+1)}^{n+\theta} - \boldsymbol{\xi}^n}{\theta \tau} = \nabla \cdot \boldsymbol{\sigma}_S(\boldsymbol{\eta}_{(\kappa+1)}^{n+\theta})}
& \eqmath[r]{B}{\mbox{ in }\Omega_S,}
\\
&
  \eqmath[l]{A}{\alpha \boldsymbol{\xi}^{n+\theta}_{(\kappa+1)} + \boldsymbol\sigma_S(\boldsymbol{\eta}^{n+\theta}_{(\kappa+1)})\boldsymbol{n}_S = \alpha \boldsymbol{u}_{(\kappa)}^{n+\theta} }
  \notag
  \\
  &
    \eqmath[l]{A}{  \qquad
    - \boldsymbol\sigma_F(\boldsymbol{u}^{n+\theta}_{(\kappa)},p^{n+\theta}_{(\kappa)})\boldsymbol{n}_F}
& \eqmath[r]{B}{\mbox{ on } \Gamma,}
\\
&
  \eqmath[l]{A}{\boldsymbol \eta^{n+\theta}_{(\kappa+1)} = 0}
    & \eqmath[r]{B}{\mbox{ on } \Gamma^{in}_S \cup \Gamma^{out}_S,}
    \\
    &
  \eqmath[l]{A}{{\boldsymbol \sigma_S (\boldsymbol{\eta}_{(\kappa+1)}^{n+\theta})}  {\boldsymbol{ n}}_S =  0 }
    & \eqmath[r]{B}{\mbox{ on } {\Gamma}_S^{ext},}
\end{empheq}
\label{eq:Dstrctr+1/2--kappa}
\end{subequations}
\begin{subequations}
\begin{empheq}[left= \ftext{2.2em}{Fluid:}  \;  {\empheqlbrace \quad}]{alignat=2}
&
 \eqmath[l]{A}{\rho_F \frac{{\boldsymbol u}^{n+\theta}_{(\kappa+1)} - {\boldsymbol u}^{n}}{\theta \tau} 
- \nabla\cdot {\boldsymbol \sigma}_F ({\boldsymbol u}^{n+\theta}_{(\kappa+1)} , p^{n+\theta}_{(\kappa+1)} )
}
\\
&
 \eqmath[l]{A}{
\qquad = \boldsymbol{f}_F(t^{n+\theta}) }
& \eqmath[r]{B}{ \mbox{ in }\Omega_F,}
\\
&
 \eqmath[l]{A}{\nabla \cdot {\boldsymbol u}^{n+\theta}_{(\kappa+1)} = 0}
& \eqmath[r]{B}{\mbox{ in }\Omega_F,}
\\
&
 \eqmath[l]{A}{\alpha \boldsymbol{u}^{n+\theta}_{(\kappa+1)} - \boldsymbol\sigma_F(\boldsymbol{u}_{(\kappa)}^{n+\theta},p_{(\kappa)}^{n+\theta}) \boldsymbol{n}_F
 = \alpha \boldsymbol{\xi}^{n+\theta}_{(\kappa+1)} 
 }
 \notag
 \\
 &
 \eqmath[l]{A}{\qquad -\boldsymbol\sigma_F(\boldsymbol{u}_{(\kappa+1)}^{n+\theta},p_{(\kappa+1)}^{n+\theta}) \boldsymbol{n}_F}
&\eqmath[r]{B}{ \mbox{ on }\Gamma,}
\\
&
 \eqmath[l]{A}{ \boldsymbol {\sigma}_F(\boldsymbol u^{n+\theta}_{(\kappa+1)}, p^{n+\theta}_{(\kappa+1)}) \boldsymbol{ n}_{F} = -p_{in} (t^{n+\theta}) \boldsymbol{n}_{F} 
}
 & 
 \eqmath[r]{B}{ \mbox{on }  \Gamma_{F}^{in},}
 \\
 &
  \eqmath[l]{A}{ \boldsymbol {\sigma}_F(\boldsymbol u^{n+\theta}_{(\kappa+1)}, p^{n+\theta}_{(\kappa+1)})  \boldsymbol{ n}_{F} = -p_{out} (t^{n+\theta}) \boldsymbol{n}_{F}
}
 & 
 \eqmath[r]{B}{ \mbox{on }  \Gamma_{F}^{out}.}
\end{empheq}
\label{eq:Dflow+1/2--kappa}
\end{subequations}
The converged solutions,
\begin{align}
{\boldsymbol \eta}^{n+\theta}_{(\kappa)} , 
{\boldsymbol \xi}^{n+\theta}_{(\kappa)},
{\boldsymbol u}^{n+\theta}_{(\kappa)} , 
p^{n+\theta}_{(\kappa)} 
\xrightarrow{\kappa\rightarrow\infty}
{\boldsymbol \eta}^{n+\theta} , 
{\boldsymbol \xi}^{n+\theta} , 
{\boldsymbol u}^{n+\theta} , 
p^{n+\theta},
\nonumber
\end{align}
then satisfy:
\begin{subequations}
\begin{empheq}[left= \ftext{2.2em}{Solid:}  \;  {\empheqlbrace} \quad]{alignat=2}
&
 \eqmath[l]{A}{\frac{{\boldsymbol \eta}^{n+\theta} - {\boldsymbol \eta}^{n}}{\theta \tau} 
= 
{\boldsymbol \xi}^{n+\theta}}
&
\eqmath[r]{B}{ \mbox{ in }\Omega_S
,}
\label{BEsolida}
\\
&
 \eqmath[l]{A}{ \rho_S \frac{\boldsymbol{\xi}^{n+\theta} - \boldsymbol{\xi}^n}{\theta \tau} = \nabla \cdot \boldsymbol{\sigma}_S(\boldsymbol{\eta}^{n+\theta})}
&
\eqmath[r]{B}{ \mbox{ in }\Omega_S
,}
\label{BEsolidb}
\\
&
 \eqmath[l]{A}{ \boldsymbol\sigma_S(\boldsymbol{\eta}^{n+\theta})\boldsymbol{n}_S =  - \boldsymbol\sigma_F(\boldsymbol{u}^{n+\theta},p^{n+\theta})\boldsymbol{n}_F}
& 
\eqmath[r]{B}{ \mbox{ on }\Gamma
,}
\label{BEsolidc}
\\
&
  \eqmath[l]{A}{\boldsymbol \eta^{n+\theta}= 0}
    & \eqmath[r]{B}{\mbox{ on } \Gamma^{in}_S \cup \Gamma^{out}_S,}
    \label{BEsolidd}
    \\
    &
  \eqmath[l]{A}{{\boldsymbol \sigma_S (\boldsymbol{\eta}^{n+\theta})}  {\boldsymbol{ n}}_S =  0 }
    & \eqmath[r]{B}{\mbox{ on } {\Gamma}_S^{ext},}
\label{BEsolide}
\end{empheq}
\label{eq:Dstrctr+1/2}
\end{subequations}
\begin{subequations}
\begin{empheq}[left= \ftext{2.2em}{Fluid:}  \;  {\empheqlbrace} \quad]{alignat=2}
&
 \eqmath[l]{A}{\rho_F \frac{{\boldsymbol u}^{n+\theta} - {\boldsymbol u}^{n}}{\theta \tau} 
- \nabla\cdot {\boldsymbol \sigma}_F ({\boldsymbol u}^{n+\theta} , p^{n+\theta})
}
\notag
\\
&
 \eqmath[l]{A}{
 \qquad
 = \boldsymbol{f}_F(t^{n+\theta}) }
&
\eqmath[r]{B}{ 
 \quad  \mbox{ in }\Omega_F,}
\label{BEfluida}
\\
&
 \eqmath[l]{A}{\nabla \cdot {\boldsymbol u}^{n+\theta} = 0}
&
\eqmath[r]{B}{ \mbox{ in }\Omega_F,}
\label{BEfluidb}
\\
&
 \eqmath[l]{A}{\boldsymbol{u}^{n+\theta}
 =  \boldsymbol{\xi}^{n+\theta}}
&
\eqmath[r]{B}{ \mbox{ on }\Gamma.}
\label{BEfluidc}
\\
&
 \eqmath[l]{A}{ \boldsymbol {\sigma}_F (\boldsymbol u^{n+\theta}, p^{n+\theta}) \boldsymbol{ n}_{F} = -p_{in} (t^{n+\theta}) \boldsymbol{n}_{F} 
}
 & 
 \eqmath[r]{B}{ \mbox{on }  \Gamma_{F}^{in},}
 \label{BEfluidd}
 \\
 &
  \eqmath[l]{A}{ \boldsymbol {\sigma}_F (\boldsymbol u^{n+\theta}, p^{n+\theta}) \boldsymbol{ n}_{F} = -p_{out} (t^{n+\theta}) \boldsymbol{n}_{F}
}
 & 
 \eqmath[r]{B}{ \mbox{on }  \Gamma_{F}^{out}.}
 \label{BEfluide}
\end{empheq}
\label{eq:Dflow+1/2}
\end{subequations}
\end{step}
\begin{step}
Now evaluate the following (equivalent to solving \textbf{Forward Euler} problems):
\begin{subequations}
\begin{empheq}[left= \ftext{2.2em}{Solid:}  \;  {\empheqlbrace} \quad]{alignat=2}
&  \eqmath[l]{A}{{\boldsymbol \eta}^{n+1} 
= \frac{1}{\theta} {\boldsymbol \eta}^{n+\theta} - \frac{1-\theta}{\theta}{\boldsymbol \eta}^{n} }
&
\eqmath[r]{B}{ 
\mbox{ in }\Omega_S,}
\label{FEextraa}
\\
&
 \eqmath[l]{A}{
{\boldsymbol \xi}^{n+1} 
= 
\frac{1}{\theta} {\boldsymbol \xi}^{n+\theta} - \frac{1-\theta}{\theta}{\boldsymbol \xi}^{n}
}
&
\eqmath[r]{B}{ 
\mbox{ in }\Omega_S,}
\label{FEextrab}
\end{empheq}
\label{eq:Dstrct+1=timefilter}
\end{subequations}
\begin{subequations}
\label{eq:Dflow+1=timefilter}
\begin{empheq}[left= \ftext{2.2em}{Fluid:}  \;  {\empheqlbrace} \quad]{alignat=2}
& 
\eqmath[l]{A}{
{\boldsymbol u}^{n+1} = \frac{1}{\theta} {\boldsymbol u}^{n+\theta} - \frac{1-\theta}{\theta} {\boldsymbol u}^{n}}
&
\eqmath[r]{B}{ \mbox{ in }\Omega_F,}
\end{empheq}
\end{subequations}
\end{step}
Set $n=n+1$, and go back to Step 1.
\end{algorithm}
\begin{remark}
\label{FE=extra}
From a computational viewpoint, the bulk of the work in Algorithm \ref{algorithm1} is performed in the BE steps \eqref{eq:Dstrctr+1/2--kappa}-\eqref{eq:Dflow+1/2--kappa},
as the FE steps \eqref{eq:Dstrct+1=timefilter}-\eqref{eq:Dflow+1=timefilter}, written as linear extrapolations, act as time-filters. 
For the theoretical argumentation, we will use their equivalent  FE form:
\begin{subequations}
\begin{empheq}[left= \ftext{2.2em}{Solid:} \;  {\empheqlbrace} \quad]{alignat=2}
&
\eqmath[l]{A}{
\frac{{\boldsymbol \eta}^{n+1} - {\boldsymbol \eta}^{n+\theta}}{(1-\theta)\tau} 
= {\boldsymbol \xi}^{n+\theta}} 
&
\eqmath[r]{B}{ \mbox{ in }\Omega_S
,}
\label{FEsolida}
\\
&
\eqmath[l]{A}{
 \rho_S \frac{\boldsymbol{\xi}^{n+1} - \boldsymbol{\xi}^{n+\theta}}{(1-\theta) \tau} = \nabla \cdot \boldsymbol{\sigma}_S(\boldsymbol{\eta}^{n+\theta})}
&
\eqmath[r]{B}{ \mbox{ in }\Omega_S
,}
\label{FEsolidb}
\end{empheq}
\label{eq:Dstrctr+1}
\end{subequations}
\begin{empheq}[left= \ftext{2.2em}{Fluid:} \;  {\empheqlbrace} \quad]{alignat=2}
&
\eqmath[l]{A}{\rho_F \frac{{\boldsymbol u}^{n+1} - {\boldsymbol u}^{n+\theta}}{(1-\theta)\tau } 
- \nabla\cdot {\boldsymbol \sigma} ({\boldsymbol u}^{n+\theta} , p^{n+\theta})}
\notag
\\
&
 \eqmath[l]{A}{
 \qquad 
= \boldsymbol{f}_F(t^{n+\theta}) }
&
\eqmath[r]{B}{ \mbox{ in }\Omega_F.}
\label{FEfluida}
\end{empheq}
\end{remark}

\begin{remark}
The BE part of Algorithm~\ref{algorithm1} is obtained by sub-iterating a loosely-coupled method proposed in~\cite{paper1}, which was shown to be unconditionally stable without sub-iterating between the fluid and structure sub-problems. However, the accuracy of the method was shown to be only $\mathcal{O}(\tau^{\frac12})$. The results obtained using Algorithm~\ref{algorithm1} and  a loosely-coupled method proposed in~\cite{paper1} are compared in Example 2 in Section~\ref{numerics}.

\end{remark}

\section{Convergence of the partitioned iterative method}
In this section, we show that the iterative method defined by~\eqref{eq:Dstrctr+1/2--kappa}-\eqref{eq:Dflow+1/2--kappa} converges. 
In the following, we will use the   polarized identity given by:
\begin{align}
2(a-c)b = a^2-c^2-(a-b)^2+(b-c)^2
\label{polarization}
\end{align}

\begin{theorem}
The sequences ${\boldsymbol u}^{n+\theta}_{(\kappa)}, {\boldsymbol \eta}^{n+\theta}_{(\kappa)} , {\boldsymbol \xi}^{n+\theta}_{(\kappa)} $
generated by the iterations \eqref{eq:Dstrctr+1/2--kappa}-\eqref{eq:Dflow+1/2--kappa} converge as $\kappa\rightarrow\infty$:
\begin{align*}
&
{\boldsymbol u}^{n+\theta}_{(\kappa)} \longrightarrow {\boldsymbol u}^{n+\theta}
\qquad\mbox{in }
\ell^\infty(H^1(\Gamma)) \cap \ell^2(L^2(\Gamma)) \cap \ell^2(H^1(\Omega_F)),
\\
& {\boldsymbol \eta}^{n+\theta}_{(\kappa)} \longrightarrow {\boldsymbol \eta}^{n+\theta}
\qquad\mbox{in }
\ell^2(S) , 
\\
& {\boldsymbol \xi}^{n+\theta}_{(\kappa)} \longrightarrow {\boldsymbol \xi}^{n+\theta}
\qquad\mbox{in }
\ell^2(L^2(\Omega_S)) \cap \ell^2(L^2(\Gamma))
.
\end{align*}
\end{theorem}
\begin{proof}
We begin by subtracting  \eqref{eq:Dstrctr+1/2--kappa}-\eqref{eq:Dflow+1/2--kappa} at iteration $\kappa$ 
from the same equations at iteration $\kappa+1$. Using notation
\begin{align}
&
\boldsymbol \delta^{\eta}_{\kappa+1}  =  {\boldsymbol \eta}^{n+\theta}_{(\kappa+1)} -{\boldsymbol \eta}^{n+\theta}_{(\kappa)},
\notag \\
&
\boldsymbol \delta^{\xi}_{\kappa+1}  =  {\boldsymbol \xi}^{n+\theta}_{(\kappa+1)} -{\boldsymbol \xi}^{n+\theta}_{(\kappa)},
\notag \\
&
\boldsymbol \delta^{u}_{\kappa+1}  =  {\boldsymbol u}^{n+\theta}_{(\kappa+1)} -{\boldsymbol u}^{n+\theta}_{(\kappa)},
\notag \\
&
 \delta^{p}_{\kappa+1}  =  {p}^{n+\theta}_{(\kappa+1)} -{p}^{n+\theta}_{(\kappa)},
\end{align}
we obtain the following:
\begin{subequations}
\begin{empheq}[left= \ftext{2.2em}{Solid:}  \;  {\empheqlbrace} \quad]{alignat=2}
&  \eqmath[l]{A}{
     \frac{\boldsymbol \delta^{\eta}_{\kappa+1}}{\theta \tau} 
= 
\boldsymbol \delta^{\xi}_{\kappa+1}}
&
\eqmath[r]{B}{ \mbox{ in }\Omega_S,}
\label{loopSBE1a}
    \\
    & 
    \eqmath[l]{A}{ \rho_S \frac{\boldsymbol \delta^{\xi}_{\kappa+1}}{\theta \tau} = \nabla \cdot \boldsymbol{\sigma}_S(\boldsymbol \delta^{\eta}_{\kappa+1})}
&
\eqmath[r]{B}{ \mbox{ in }\Omega_S,}
\label{loopSBE1b}
    \\
    &
 \eqmath[l]{A}{\alpha \boldsymbol \delta^{\xi}_{\kappa+1} + \boldsymbol\sigma_S(\boldsymbol \delta^{\eta}_{\kappa+1})\boldsymbol{n}_S 
= \alpha \boldsymbol \delta^{u}_{\kappa} }
&
\notag
\\
&
 \eqmath[l]{A}{\qquad - \boldsymbol\sigma_F(\boldsymbol \delta^{u}_{\kappa},  \delta^{p}_{\kappa})\boldsymbol{n}_F}
& 
\eqmath[r]{B}{\mbox{ on }\Gamma,}
\label{loopSBE1c}
\\
& \eqmath[l]{A}{\boldsymbol \delta^{\eta}_{\kappa+1} = 0}
 & 
 \eqmath[r]{B}{\mbox{ on } \Gamma^{in}_S \cup \Gamma^{out}_S,}
\label{loopSBE1d}
\\
    &
  \eqmath[l]{A}{{\boldsymbol \sigma_S (\boldsymbol \delta^{\eta}_{\kappa+1})}  {\boldsymbol{ n}}_S =  0 }
    & \eqmath[r]{B}{\mbox{ on } {\Gamma}_S^{ext},}
    \label{loopSBE1e}
\end{empheq}
\end{subequations}
\begin{subequations}
\begin{empheq}[left= \ftext{2.2em}{Fluid:}  \;  {\empheqlbrace} \quad]{alignat=2}
&  \eqmath[l]{A}{
\rho_F \frac{\boldsymbol \delta^{u}_{\kappa+1}}{\theta \tau} 
- \nabla\cdot {\boldsymbol \sigma}_F (\boldsymbol \delta^{u}_{\kappa+1},  \delta^{p}_{\kappa+1} )
= 0 }
&
 \eqmath[r]{B}{\mbox{ in }\Omega_F,}
\label{loopFBE1a}
\\
& 
\eqmath[l]{A}{
\nabla \cdot \boldsymbol \delta^{u}_{\kappa+1}= 0}
&
 \eqmath[r]{B}{\mbox{ in }\Omega_F,}
\label{loopFBE1b}
\\
&
\eqmath[l]{A}{
\alpha \boldsymbol \delta^{u}_{\kappa+1}
 - \boldsymbol\sigma_F(\boldsymbol \delta^{u}_{\kappa},  \delta^{p}_{\kappa}) \boldsymbol{n}_F}
 \notag
\\
&
\eqmath[l]{A}{\qquad 
 = \alpha \boldsymbol \delta^{\xi}_{\kappa+1}
  -\boldsymbol\sigma_F(\boldsymbol \delta^{u}_{\kappa+1},  \delta^{p}_{\kappa+1}) \boldsymbol{n}_F}
&
 \eqmath[r]{B}{\mbox{ on }\Gamma,}
\label{loopFBE1c}
\\
&
 \eqmath[l]{A}{ \boldsymbol {\sigma}_F(\boldsymbol \delta^{u}_{\kappa+1},  \delta^{p}_{\kappa+1}) \boldsymbol{ n}_{F} = 0
}
 & 
 \eqmath[r]{B}{ \mbox{on }  \Gamma_{F}^{in},}
 \\
 &
  \eqmath[l]{A}{ \boldsymbol {\sigma}_F(\boldsymbol \delta^{u}_{\kappa+1},  \delta^{p}_{\kappa+1})  \boldsymbol{ n}_{F} = 0
}
 & 
 \eqmath[r]{B}{ \mbox{on }  \Gamma_{F}^{out}.}
\end{empheq}
\end{subequations}
We  multiply~\eqref{loopSBE1b} by  $\boldsymbol \delta^{\xi}_{\kappa+1}$ and  integrate over $\Omega_S$. Using~\eqref{loopSBE1a} and~\eqref{sNorm}, we have:
\begin{align}
0&=\frac{\rho_S}{\theta \tau} \left\| \boldsymbol \delta^{\xi}_{\kappa+1} \right\|_{L^2(\Omega_S)}^2
+ \frac{1}{\theta \tau}  \left\| \boldsymbol \delta^{\eta}_{\kappa+1} \right\|_{S}^2 
- \int_{\Gamma}  \boldsymbol\sigma_S(\boldsymbol \delta^{\eta}_{\kappa+1}) \boldsymbol{n}_S 
\cdot \boldsymbol \delta^{\xi}_{\kappa+1}.
\notag
\end{align}
Using condition~\eqref{loopSBE1c} and identity~\eqref{polarization}, we have:
\begin{align}
0&=\frac{\rho_S}{\theta \tau} \left\| \boldsymbol \delta^{\xi}_{\kappa+1} \right\|_{L^2(\Omega_S)}^2
+ \frac{1}{\theta \tau}  \left\| \boldsymbol \delta^{\eta}_{\kappa+1} \right\|_{S}^2 
+\frac{  \alpha}{2} \left\| \boldsymbol \delta^{\xi}_{\kappa+1} \right\|^2_{L^2(\Gamma)} 
 -\frac{  \alpha}{2} \left\| \boldsymbol \delta^{u}_{\kappa} \right\|^2_{L^2(\Gamma)}
 \notag \\
 &
  +\frac{  \alpha}{2} \left\| \boldsymbol \delta^{\xi}_{\kappa+1} -\boldsymbol \delta^{u}_{\kappa} \right\|^2_{L^2(\Gamma)}
+\int_{\Gamma} \boldsymbol\sigma_F(\boldsymbol \delta^{u}_{\kappa}, \delta^{p}_{\kappa})\boldsymbol{n}_F
\cdot \boldsymbol \delta^{\xi}_{\kappa+1}.
\label{loopSBE}
\end{align}
We address the fluid in a similar manner. Multiplying~\eqref{loopFBE1a} by $\boldsymbol \delta^{u}_{\kappa+1}$,~\eqref{loopFBE1b} by $\boldsymbol \delta^{p}_{\kappa+1}$, integrating over $\Omega_F$ and adding the resulting equations together, we obtain:
\begin{align}
0=& \frac{\rho_F}{\theta \tau} \left\| \boldsymbol \delta^{u}_{\kappa+1} \right\|^2_{L^2(\Omega_F)} 
+2 \mu_F \left\| \boldsymbol{D}(\boldsymbol \delta^{u}_{\kappa+1}) \right\|_{L^2(\Omega_F)}^2
-\int_{\Gamma} \boldsymbol\sigma_F (\boldsymbol \delta^{u}_{\kappa+1},  \delta^{p}_{\kappa+1})\boldsymbol{n}_F  \cdot \boldsymbol \delta^{u}_{\kappa+1}.
\notag
\end{align}
Using Robin condition~\eqref{loopFBE1c} and identity~\eqref{polarization}, we have:
\begin{align}
0=& \frac{\rho_F}{\theta \tau} \left\| \boldsymbol \delta^{u}_{\kappa+1}  \right\|^2_{L^2(\Omega_F)} 
+2 \mu_F \left\| \boldsymbol{D}(\boldsymbol \delta^{u}_{\kappa+1} )\right\|_{L^2(\Omega_F)}^2
+\frac{\alpha}{2} \left\| \boldsymbol \delta^{u}_{\kappa+1}  \right\|^2_{L^2(\Gamma)}
\nonumber \\
&
-\frac{\alpha}{2} \left\| \boldsymbol \delta^{\xi}_{\kappa+1} \right\|^2_{L^2(\Gamma)} 
+\frac{\alpha}{2} \left\| \boldsymbol \delta^{u}_{\kappa+1}  
-\boldsymbol \delta^{\xi}_{\kappa+1} \right\|^2_{L^2(\Gamma)} 
\notag \\
& -\int_{\Gamma}  \boldsymbol\sigma_F(\boldsymbol \delta^{u}_{\kappa},  \delta^{p}_{\kappa}) \boldsymbol{n}_F \cdot \boldsymbol \delta^{u}_{\kappa+1}.
\label{loopFBE}
\end{align}
Combining structure~\eqref{loopSBE} and fluid~\eqref{loopFBE} estimates, we obtain:
\begin{align}
0&=\frac{\rho_S}{\theta \tau} \left\|  \boldsymbol \delta^{\xi}_{\kappa+1} \right\|_{L^2(\Omega_S)}^2
+ \frac{1}{\theta \tau}  \left\| \boldsymbol \delta^{\eta}_{\kappa+1}   \right\|_{S}^2 
+\frac{\rho_F}{\theta \tau} \left\|\boldsymbol \delta^{u}_{\kappa+1}   \right\|^2_{L^2(\Omega_F)} 
\notag \\
& 
+2 \mu_F \left\| \boldsymbol{D}(\boldsymbol \delta^{u}_{\kappa+1}  )\right\|_{L^2(\Omega_F)}^2
+\frac{\alpha}{2} \left\| \boldsymbol \delta^{u}_{\kappa+1}  \right\|^2_{L^2(\Gamma)}
-\frac{  \alpha}{2} \left\| \boldsymbol \delta^{u}_{\kappa}  \right\|^2_{L^2(\Gamma)}
\notag \\
& 
  +\frac{  \alpha}{2} \left\| \boldsymbol \delta^{\xi}_{\kappa+1} -\boldsymbol \delta^{u}_{\kappa} \right\|^2_{L^2(\Gamma)}
+\frac{\alpha}{2} \left\| \boldsymbol \delta^{u}_{\kappa+1} -\boldsymbol \delta^{\xi}_{\kappa+1}  \right\|^2_{L^2(\Gamma)} 
\notag \\
&+\int_{\Gamma} \boldsymbol\sigma_F(\boldsymbol \delta^{u}_{\kappa} ,  \delta^{p}_{\kappa} )\boldsymbol{n}_F \cdot
 \left(\boldsymbol \delta^{\xi}_{\kappa+1}   -\boldsymbol \delta^{u}_{\kappa+1}  \right).
 \label{fs_est}
\end{align}
Using~\eqref{loopFBE1c} and~\eqref{polarization}, the last term can be written as:
\begin{align}
&\int_{\Gamma} \boldsymbol\sigma_F(\boldsymbol \delta^{u}_{\kappa} ,  \delta^{p}_{\kappa} )\boldsymbol{n}_F \cdot
 \left(\boldsymbol \delta^{\xi}_{\kappa+1}   -\boldsymbol \delta^{u}_{\kappa+1}  \right)
 \notag \\
 &=
 \frac{1}{2 \alpha } \left\|  \boldsymbol\sigma_F(\boldsymbol \delta^{u}_{\kappa+1} ,  \delta^{p}_{\kappa+1} ) \boldsymbol{n}_F \right\|^2_{L^2(\Gamma)}
 -\frac{1}{2 \alpha } \left\| \boldsymbol\sigma_F(\boldsymbol \delta^{u}_{\kappa} ,  \delta^{p}_{\kappa} ) \boldsymbol{n}_F \right\|^2_{L^2(\Gamma)}
 \notag \\
 &-\frac{1}{2 \alpha } \left\|  \boldsymbol\sigma_F(\boldsymbol \delta^{u}_{\kappa+1} ,  \delta^{p}_{\kappa+1} )\boldsymbol{n}_F
-\boldsymbol\sigma_F(\boldsymbol \delta^{u}_{\kappa} ,  \delta^{p}_{\kappa} ) \boldsymbol{n}_F \right\|^2_{L^2(\Gamma)}.
\label{stress_est}
\end{align}
Using~\eqref{loopFBE1c} again, and combining~\eqref{stress_est} with~\eqref{fs_est}, we obtain:
\begin{align}
0&=\frac{\rho_S}{\theta \tau} \left\|  \boldsymbol \delta^{\xi}_{\kappa+1}  \right\|_{L^2(\Omega_S)}^2
+ \frac{1}{\theta \tau}  \left\| \boldsymbol \delta^{\eta}_{\kappa+1}  \right\|_{S}^2 
+\frac{\rho_F}{\theta \tau} \left\| \boldsymbol \delta^{u}_{\kappa+1}   \right\|^2_{L^2(\Omega_F)} 
\notag \\
& 
+2 \mu_F \left\| \boldsymbol{D}(\boldsymbol \delta^{u}_{\kappa+1} )\right\|_{L^2(\Omega_F)}^2
+\frac{\alpha}{2} \left\| \boldsymbol \delta^{u}_{\kappa+1}  \right\|^2_{L^2(\Gamma)}
-\frac{  \alpha}{2} \left\| \boldsymbol \delta^{u}_{\kappa} \right\|^2_{L^2(\Gamma)}
\notag \\
& 
  +\frac{  \alpha}{2} \left\| \boldsymbol \delta^{\xi}_{\kappa+1}  -\boldsymbol \delta^{u}_{\kappa} \right\|^2_{L^2(\Gamma)}
+ \frac{1}{2\alpha } \left\|  \boldsymbol\sigma_F(\boldsymbol \delta^{u}_{\kappa+1} ,   \delta^{p}_{\kappa+1} ) \boldsymbol{n}_F \right\|^2_{L^2(\Gamma)}
\notag 
\\
&
 -\frac{1}{2\alpha } \left\| \boldsymbol\sigma_F(\boldsymbol \delta^{u}_{\kappa} ,  \delta^{p}_{\kappa} )\boldsymbol{n}_F \right\|^2_{L^2(\Gamma)}.
\end{align}
Summing from $\kappa=1$ to $l-1$, we get:
\begin{align}
& \frac{\rho_S}{\theta \tau} \sum_{\kappa=1}^{l-1} \left\|   \boldsymbol \delta^{\xi}_{\kappa+1} \right\|_{L^2(\Omega_S)}^2 
+ \frac{1}{\theta \tau}   \sum_{\kappa=1}^{l-1} \left\|  \boldsymbol \delta^{\eta}_{\kappa+1} \right\|_{S}^2 
+\frac{\rho_F}{\theta \tau} \sum_{\kappa=1}^{l-1} \left\|  \boldsymbol \delta^{u}_{\kappa+1} \right\|^2_{L^2(\Omega_F)} 
\nonumber\\
&+2 \mu_F \sum_{\kappa=1}^{l-1} \left\| \boldsymbol{D}( \boldsymbol \delta^{u}_{\kappa+1}) \right\|_{L^2(\Omega_F)}^2
  +\frac{  \alpha}{2}  \sum_{\kappa=1}^{l-1} \left\| \boldsymbol \delta^{\xi}_{\kappa+1} - \boldsymbol \delta^{u}_{\kappa} \right\|^2_{L^2(\Gamma)}
 \nonumber \\
&+ \frac{\alpha}{2} \left\|
 \boldsymbol \delta^{u}_{l} \right\|^2_{L^2(\Gamma)}
+ \frac{1}{2\alpha } \left\|  \boldsymbol\sigma_F( \boldsymbol \delta^{u}_{l},   \delta^{p}_{l}) \boldsymbol{n}_F \right\|^2_{L^2(\Gamma)}
 \nonumber \\
&=
\frac{\alpha}{2} \left\|
 \boldsymbol \delta^{u}_{1} \right\|^2_{L^2(\Gamma)}
+ \frac{1}{2\alpha } \left\|  \boldsymbol\sigma_F( \boldsymbol \delta^{u}_{1},  \delta^{p}_{1}) \boldsymbol{n}_F \right\|^2_{L^2(\Gamma)}.
\end{align}

Hence, $\boldsymbol{\eta}_{(\kappa)}^{n+\theta}$ and $\boldsymbol{\xi}_{(\kappa)}^{n+\theta}$, and $\boldsymbol{u}_{(\kappa)}^{n+\theta}$ are Cauchy sequences in $\ell^2(S), \ell^2(L^2(\Omega_S)) \cap \ell^2(L^2(\Gamma))$ and $\ell^{\infty}(H^1(\Gamma))\cap \ell^{2}(L^2(\Gamma)) \cap \ell^{2}(H^1(\Omega_F))$, respectively. The completeness of the spaces implies the convergence of the
iterations, completing the proof.
\end{proof}

\section{Stability Analysis}
In this section, we prove the stability of the partitioned method presented in Algorithm~\ref{algorithm1}. In particular, we consider the scheme described by the BE steps~\eqref{eq:Dstrctr+1/2}-\eqref{eq:Dflow+1/2} and FE steps~\eqref{eq:Dstrctr+1}-\eqref{FEfluida}. As noted in Remark~\ref{FE=extra}, the FE steps are equivalent to linear extrapolations~\eqref{eq:Dstrct+1=timefilter}-\eqref{eq:Dflow+1=timefilter}. 

Let $\mathcal{E}^n$ denote the sum of the kinetic and elastic energy of the solid, and kinetic energy of the fluid, defined as:
\begin{align*}
\mathcal{E}^n = \frac{\rho_S}{2 }  \|   \boldsymbol{\xi}^{n}\|_{L^2(\Omega_S)}^2
+\frac{1}{2 }  \| \boldsymbol{\eta}^{n} \|_{S}^2
+\frac{\rho_F}{2 }  \|\boldsymbol{u}^{n} \|_{L^2(\Omega_F)}^2,
\end{align*}
let $\mathcal{D}^n$ denote the fluid viscous dissipation, given by:
\begin{align*}
\mathcal{D}^n=\mu_F \tau \sum_{k=2}^{n-1}\|\boldsymbol{D}(\boldsymbol{u}^{k+\theta}) \|_{L^2(\Omega_F)}^2, 
\end{align*}
 let $\mathcal{N}^n$ denote the terms present due to numerical dissipation:
\begin{align*}
\mathcal{N}^n = &
\frac{ (2 \theta -1)}{2 \tau} \sum_{k=2}^{n-1} \left( 
\rho_S  \left\|   \boldsymbol{\xi}^{k+1} -\boldsymbol{\xi}^k  \right\|_{L^2(\Omega_S)}^2
+ 
  \left\|   \boldsymbol{\eta}^{k+1} -\boldsymbol{\eta}^k  \right\|_{S}^2
\right)
\notag \\
&
+  \frac{\rho_F (2\theta-1)}{2 \tau} \sum_{k=2}^{n-1} \|\boldsymbol{u}^{k+1} - \boldsymbol{u}^{k}\|^2_{L^2(\Omega_F)},
\end{align*}
and let $\mathcal{F}^n$ denote the forcing terms:
\begin{align*}
\mathcal{F}^n =& \frac{ \tau}{\mu_F}\sum_{k=2}^{n-1} \left( C_1 \|  \boldsymbol f_F(t^{k+\theta})\|^2_{L^2(\Omega_F)}
+
C_2 \| p_{in}(t^{k+\theta}) \|^2_{L^2(\Gamma_F^{in})}
\right)
\notag
\\
&
+
\frac{ \tau}{\mu_F}\sum_{k=2}^{n-1} C_2\| p_{out}(t^{k+\theta}) \|^2_{L^2(\Gamma_F^{out})}
.
\end{align*}
The stability result is given in the following theorem. 
\begin{theorem}
Let $\{(\boldsymbol \xi^n, \boldsymbol \eta^n, \boldsymbol u^n, p^n)\}_{2 \leq n \leq N}$ be the solution of Algorithm~\ref{algorithm1}.
Assume that  $\theta \in [\frac12, 1]$. Then, the following estimate holds:
\begin{align}
&
\mathcal{E}^N+\mathcal{D}^N+\mathcal{N}^N 
\leq
\mathcal{E}^2
+\mathcal{F}^N.
\end{align}

\end{theorem}

\begin{proof}
We multiply~\eqref{BEsolidb} by $\theta \boldsymbol{\xi}^{n+\theta}$, integrate over $\Omega_S$, and use~\eqref{BEsolida} and~\eqref{polarization}, which yields:
\begin{align}
&0=  \frac{\rho_S}{2 \tau} \left(  \|\boldsymbol{\xi}^{n+\theta} \|_{L^2(\Omega_S)}^2 - \|  \boldsymbol{\xi}^{n} \|_{L^2(\Omega_S)}^2 + \|  \boldsymbol{\xi}^{n+\theta} -\boldsymbol{\xi}^n \|_{L^2(\Omega_S)}^2 \right) 
\nonumber \\
&+\frac{1}{2  \tau} \left( \| \boldsymbol{\eta}^{n+\theta}\|_{S}^2 - \|\boldsymbol{\eta}^{n} \|_{S}^2 + \| \boldsymbol{\eta}^{n+\theta} - \boldsymbol{\eta}^{n} \|_{S}^2 \right) 
 - \theta \int_{\Gamma} \boldsymbol\sigma_S(\boldsymbol{\eta}^{n+\theta}) \boldsymbol{n}_S  \boldsymbol{\xi}^{n+\theta}. \label{SBE}
\end{align}
Similarly, we multiply~\eqref{FEsolidb} by $(1-\theta) \boldsymbol{\xi}^{n+\theta}$, integrate over $\Omega_S$ and use~\eqref{FEsolida} and~\eqref{polarization} in order to obtain:
\begin{align}
0=& \frac{\rho_S}{2 \tau} \left( \|  \boldsymbol{\xi}^{n+1} \|_{L^2(\Omega_S)}^2-\|  \boldsymbol{\xi}^{n+\theta} \|_{L^2(\Omega_S)}^2  - \|  \boldsymbol{\xi}^{n+\theta} -\boldsymbol{\xi}^{n+1} \|_{L^2(\Omega_S)}^2 \right)
 \nonumber \\
&+\frac{1}{2 \tau} \left( 
 \| \boldsymbol{\eta}^{n+1} \|_{S}^2
 - \| \boldsymbol{\eta}^{n+\theta} \|_{S}^2
- \| \boldsymbol{\eta}^{n+\theta} - \boldsymbol{\eta}^{n+1} \|_{S}^2 
\right) 
\notag
\\
&
-(1-\theta) \int_{\Gamma} \boldsymbol\sigma_S(\boldsymbol{\eta}^{n+\theta}) \boldsymbol{n}_S \cdot \boldsymbol{\xi}^{n+\theta}.
\label{SFE}
\end{align}
Adding~\eqref{SBE} and~\eqref{SFE}, and using~\eqref{BEsolidc}, we have:
\begin{align}
0=&
\frac{\rho_S}{2\tau} \left( \|   \boldsymbol{\xi}^{n+1}\|_{L^2(\Omega_S)}^2 - \|  \boldsymbol{\xi}^{n} \|_{L^2(\Omega_S)}^2 + \|  \boldsymbol{\xi}^{n+\theta} -\boldsymbol{\xi}^n \|_{L^2(\Omega_S)}^2 \right)
\nonumber 
\\
&
- \frac{\rho_S}{2\tau}\|  \boldsymbol{\xi}^{n+\theta} -\boldsymbol{\xi}^{n+1} \|_{L^2(\Omega_S)}^2 
+\frac{1}{2 \tau} \left( 
 \| \boldsymbol{\eta}^{n+1} \|_{S}^2
 - \| \boldsymbol{\eta}^{n} \|_{S}^2
 + \| \boldsymbol{\eta}^{n+\theta} - \boldsymbol{\eta}^{n} \|_{S}^2 
 \right)
\nonumber 
\\
&-\frac{1}{2 \tau} 
 \| \boldsymbol{\eta}^{n+\theta} - \boldsymbol{\eta}^{n+1} \|_{S}^2 
 + \int_{\Gamma}  \boldsymbol\sigma_F(\boldsymbol{u}^{n+\theta},p^{n+\theta})\boldsymbol{n}_F \cdot \boldsymbol{\xi}^{n+\theta}.
   \label{SBEFEtemp}
\end{align}
Using~\eqref{FEextrab}, we have:
\begin{align}
&\|  \boldsymbol{\xi}^{n+\theta} -\boldsymbol{\xi}^n \|_{L^2(\Omega_S)}^2 - \|  \boldsymbol{\xi}^{n+\theta} -\boldsymbol{\xi}^{n+1} \|_{L^2(\Omega_S)}^2 
\notag \\
&\qquad 
=(2 \theta -1)  \left\|   \boldsymbol{\xi}^{n+1} -\boldsymbol{\xi}^n  \right\|_{L^2(\Omega_S)}^2,
\label{xiterms}
\end{align}
noting that $(2 \theta -1) \geq 0$ since $\theta \in [\frac12, 1]$.
Similarly, using~\eqref{FEextraa}, we can write:
\begin{align}
  \| \boldsymbol{\eta}^{n+\theta} - \boldsymbol{\eta}^{n} \|_{S}^2 
- \| \boldsymbol{\eta}^{n+\theta} - \boldsymbol{\eta}^{n+1} \|_{S}^2
= (2 \theta -1)  \left\|   \boldsymbol{\eta}^{n+1} -\boldsymbol{\eta}^n  \right\|_{S}^2.
 \label{etaterms}
\end{align}
Using~\eqref{xiterms} and~\eqref{etaterms}, the solid estimate~\eqref{SBEFEtemp} becomes:
\begin{align}
0=&
\frac{\rho_S}{2\tau} \left( \|   \boldsymbol{\xi}^{n+1}\|_{L^2(\Omega_S)}^2 - \|  \boldsymbol{\xi}^{n} \|_{L^2(\Omega_S)}^2 \right)
+\frac{\rho_S (2 \theta -1) }{2 \tau}  \left\|   \boldsymbol{\xi}^{n+1} -\boldsymbol{\xi}^n  \right\|_{L^2(\Omega_S)}^2
 \nonumber \\
&
+\frac{1}{2 \tau} \left( 
 \| \boldsymbol{\eta}^{n+1} \|_{S}^2
 - \| \boldsymbol{\eta}^{n} \|_{S}^2
 \right)
+  \frac{  (2\theta-1)}{2 \tau}   
  \left\|   \boldsymbol{\eta}^{n+1} -\boldsymbol{\eta}^n  \right\|_{S}^2
\notag
\\
&
 + \int_{\Gamma}  \boldsymbol\sigma_F(\boldsymbol{u}^{n+\theta},p^{n+\theta})\boldsymbol{n}_F \cdot \boldsymbol{\xi}^{n+\theta}.
   \label{SBEFE}
\end{align}
In a similar way, to derive an estimate for the fluid part we multiply~\eqref{BEfluida} by $\theta \boldsymbol{u}^{n+\theta}$,~\eqref{BEfluidb} by $ p^{n+\theta}$, and~\eqref{FEfluida} by $(1-\theta )\boldsymbol{u}^{n+\theta}$,  add together and integrate over $\Omega_F$, which results in:
\begin{align*}
&
\frac{\rho_F}{2 \tau} \left( \|\boldsymbol{u}^{n+1} \|_{L^2(\Omega_F)}^2 - \|\boldsymbol{u}^{n}\|_{L^2(\Omega_F)}^2 + \|\boldsymbol{u}^{n+\theta} -\boldsymbol{u}^{n}\|_{L^2(\Omega_F)}^2\right) 
\nonumber \\
& 
 -\frac{\rho_F}{2 \tau} \|\boldsymbol{u}^{n+\theta} -\boldsymbol{u}^{n+1}\|_{L^2(\Omega_F)}^2
 +2\mu_F \|\boldsymbol{D}(\boldsymbol{u}^{n+\theta}) \|_{L^2(\Omega_F)}^2 
 \nonumber \\
&
= \int_{\Gamma} \boldsymbol\sigma_F(\boldsymbol{u}^{n+\theta},p^{n+\theta})\boldsymbol{n}_F\cdot \boldsymbol{u}^{n+\theta}
+\int_{\Omega_F} \boldsymbol f_F (t^{n+\theta}) \cdot \boldsymbol{u}^{n+\theta}
\nonumber \\
&
+\int_{\Gamma_F^{in}} p_{in}(t^{n+\theta}) \boldsymbol{u}^{n+\theta} \cdot \boldsymbol{n}_F
+\int_{\Gamma_F^{out}} p_{out}(t^{n+\theta}) \boldsymbol{u}^{n+\theta} \cdot \boldsymbol{n}_F.
\end{align*}
Note that using~\eqref{eq:Dflow+1=timefilter}, we have:
\begin{align*}
\| \boldsymbol{u}^{n+\theta} - \boldsymbol{u}^{n} \|_{L^2(\Omega_F)}^2 
-\| \boldsymbol{u}^{n+\theta} - \boldsymbol{u}^{n+1} \|_{L^2(\Omega_F)}^2 
= (2\theta-1)  \|\boldsymbol{u}^{n+1} - \boldsymbol{u}^{n}\|^2_{L^2(\Omega_F)}.
\end{align*}
Hence, the estimate for the fluid problem reads as follows:
\begin{align}
&
\frac{\rho_F}{2 \tau} \left( \|\boldsymbol{u}^{n+1} \|_{L^2(\Omega_F)}^2 - \|\boldsymbol{u}^{n}\|_{L^2(\Omega_F)}^2\right)
+ \frac{\rho_F (2\theta-1)}{2 \tau} \|\boldsymbol{u}^{n+1} - \boldsymbol{u}^{n}\|^2_{L^2(\Omega_F)}
 \nonumber \\
& +2\mu_F \|\boldsymbol{D}(\boldsymbol{u}^{n+\theta}) \|_{L^2(\Omega_F)}^2 
=\int_{\Gamma} \boldsymbol\sigma_F(\boldsymbol{u}^{n+\theta},p^{n+\theta})\boldsymbol{n}_F \cdot \boldsymbol{u}^{n+\theta}
\nonumber \\
&
+\int_{\Omega_F} \boldsymbol f_F (t^{n+\theta}) \cdot \boldsymbol{u}^{n+\theta}
+\int_{\Gamma_F^{in}} p_{in}(t^{n+\theta}) \boldsymbol{u}^{n+\theta} \cdot \boldsymbol{n}_F
\nonumber \\
&
+\int_{\Gamma_F^{out}} p_{out}(t^{n+\theta}) \boldsymbol{u}^{n+\theta} \cdot \boldsymbol{n}_F.
 \label{FBEFE}
\end{align}
Combining solid~\eqref{SBEFE} and fluid~\eqref{FBEFE} estimates and using~\eqref{BEfluidc}, we obtain:
\begin{align}
&\frac{\rho_S}{2 \tau} \left( \|   \boldsymbol{\xi}^{n+1}\|_{L^2(\Omega_S)}^2 - \|  \boldsymbol{\xi}^{n} \|_{L^2(\Omega_S)}^2 \right)
+\frac{\rho_S (2 \theta -1) }{2 \tau}  \left\|   \boldsymbol{\xi}^{n+1} -\boldsymbol{\xi}^n  \right\|_{L^2(\Omega_S)}^2
\nonumber \\
&
+\frac{1}{2 \tau} \left( 
 \| \boldsymbol{\eta}^{n+1} \|_{S}^2
 - \| \boldsymbol{\eta}^{n} \|_{S}^2
 \right)
+  \frac{  (2\theta-1)}{2 \tau}   
  \left\|   \boldsymbol{\eta}^{n+1} -\boldsymbol{\eta}^n  \right\|_{S}^2
+\frac{\rho_F}{2 \tau} \|\boldsymbol{u}^{n+1} \|_{L^2(\Omega_F)}^2 
\nonumber \\
&
-\frac{\rho_F}{2 \tau}  \|\boldsymbol{u}^{n}\|_{L^2(\Omega_F)}^2
+ \frac{\rho_F (2\theta-1)}{2 \tau} \|\boldsymbol{u}^{n+1} - \boldsymbol{u}^{n}\|^2_{L^2(\Omega_F)}
 +2\mu_F \|\boldsymbol{D}(\boldsymbol{u}^{n+\theta}) \|_{L^2(\Omega_F)}^2 
 \nonumber \\
&
=\int_{\Omega_F} \boldsymbol f_F (t^{n+\theta}) \cdot \boldsymbol{u}^{n+\theta}
+\int_{\Gamma_F^{in}} p_{in}(t^{n+\theta}) \boldsymbol{u}^{n+\theta} \cdot \boldsymbol{n}_F
\notag 
\\
&
+\int_{\Gamma_F^{out}} p_{out}(t^{n+\theta}) \boldsymbol{u}^{n+\theta} \cdot \boldsymbol{n}_F.
\label{BEFEcomb}
\end{align}
Using the Cauchy-Schwarz, Trace, Poincar\'e and Korn inequalities~\cite{bukavc2012fluid}, we can estimate:
\begin{align}
&\int_{\Omega_F} \boldsymbol f_F (t^{n+\theta}) \cdot \boldsymbol{u}^{n+\theta}
+\int_{\Gamma_F^{in}} p_{in}(t^{n+\theta}) \boldsymbol{u}^{n+\theta} \cdot \boldsymbol{n}_F
+\int_{\Gamma_F^{out}} p_{out}(t^{n+\theta}) \boldsymbol{u}^{n+\theta} \cdot \boldsymbol{n}_F
\notag 
\\
&\leq
\frac{C_1  }{\mu_F} \|  \boldsymbol f_F(t^{n+\theta})\|^2_{L^2(\Omega_F)}
+
\frac{C_2 }{\mu_F} \| p_{in}(t^{n+\theta}) \|^2_{L^2(\Gamma_F^{in})}
+
\frac{C_2 }{\mu_F} \| p_{out}(t^{n+\theta}) \|^2_{L^2(\Gamma_F^{out})}
\notag \\
&
 +\mu_F \|\boldsymbol{D}(\boldsymbol{u}^{n+\theta}) \|_{L^2(\Omega_F)}^2,
 \label{Sforcing}
\end{align}
where $C_1$ and $C_2$ do not depend on the time-discretization parameter $\tau$.
Combining~\eqref{Sforcing} with~\eqref{BEFEcomb}, summing from $n=2$ to $N-1$ and multiplying by $\tau$ yields the desired estimate.

\end{proof}

\if 1=0

\section{Extension to moving domain fluid-structure interaction}\label{movingDomain}
In this section, we extend the model and numerical scheme presented in Algorithm~\ref{algorithm1} to describe FSI in a moving domain. 
We assume that the structure equations are given in a Lagrangian framework, with respect to a reference domain $\hat{\Omega}_S$. The fluid
and structure domains at time $t$ will be denoted as $\Omega_F(t)$ and $\Omega_S(t)$, respectively. 
The fluid equations will be described in the ALE formulation.

To track the deformation of the fluid domain in time, we introduce a smooth, invertible, ALE mapping ${\mathcal{A}}: \hat{\Omega}_F \times [0,T] \rightarrow \Omega_F(t)$ given by
\begin{equation*}
{\mathcal{A}} ({\boldsymbol X},t)=  {\boldsymbol X} + {\boldsymbol \eta}_F({\boldsymbol X},t), \quad \; \textrm{for all } {\boldsymbol X} \in \hat{\Omega}_F, t \in [0,T],
\label{ale}
\end{equation*}
where ${\boldsymbol \eta}_F$ denotes the  displacement of the fluid domain. We denote the fluid deformation gradient by ${\boldsymbol F} = {\nabla}\mathcal{A}$ and its determinant by ${J}$.  We compute ${\boldsymbol \eta}_F$ using as a harmonic extension, defined by~\cite{langer2018numerical}:
\begin{align}
&- \Delta {\boldsymbol \eta}^{n+1}_F = 0 & \textrm{in} \; \hat{\Omega}_F,  \\
&  {\boldsymbol \eta}^{n+1}_F = 0 & \textrm{on} \; \hat{\Gamma}^{in}_F \cup \hat{\Gamma}^{out}_F, \\
&  {\boldsymbol \eta}^{n+1}_F = {\boldsymbol \eta}^{n+1} & \textrm{on} \; \hat{\Gamma},
\end{align}
To simplify the notation, we will write
$$
\int_{\Omega(t^{m})} \boldsymbol v^{n}
\qquad 
\textrm{instead of}
\quad
\int_{\Omega(t^{m})} \boldsymbol v^{n} \circ \mathcal{A}(t^{n}) \circ \mathcal{A}^{-1}(t^{m})
$$
whenever we need to integrate $\boldsymbol v^n$ on a domain $\Omega(t^{m})$, for $m \neq n$.

To model the fluid flow, we extend the previous model and consider the Navier-Stokes equations in the ALE form~\cite{langer2018numerical,multilayered,thick}, given by:
\begin{subequations}
\begin{align}\label{NSale1moving}
& \rho_F \left(  \partial_t \boldsymbol{u} |_{\hat{\Omega}_F}+ (\boldsymbol{u}-\boldsymbol{w}) \cdot \nabla \boldsymbol{u} \right)  = \nabla \cdot \boldsymbol\sigma_F(\boldsymbol u, p) + \boldsymbol f_F&  \textrm{in}\; \Omega_F(t)\times(0,T), \\
 \label{NSale2moving}
&\nabla \cdot \boldsymbol{u} = 0 & \textrm{in}\; \Omega_F(t)\times(0,T),
\end{align}
\end{subequations}
where $\boldsymbol w= \partial_t \boldsymbol x|_{\hat{\Omega}_F} = \partial_t \mathcal{A} \circ \mathcal{A}^{-1}$ is the domain velocity. We note that $ \partial_t \boldsymbol{u} |_{\hat{\Omega}_F}$ denotes  the Eulerian description of the ALE field $\partial_t {\boldsymbol{u}} \circ \mathcal{A}$~\cite{formaggia2010cardiovascular}, {\emph i.e.},
\begin{equation*}
 \partial_t \boldsymbol{u}(\boldsymbol x,t) |_{\hat{\Omega}_F} = \partial_t \boldsymbol{u}(\mathcal{A}^{-1}(\boldsymbol x,t),t).
\end{equation*}
The fluid domain is determined by
$
\Omega_F(t) = \mathcal{A}(\hat{\Omega}_F, t).
$
The structure model remains the same as in~\eqref{solid}.
Hence, the fully-coupled, moving domain FSI problem is given by:
\begin{subequations}
\begin{align}\label{fsi1}
& \rho_F \left( \partial_t \boldsymbol{u} |_{\hat{\Omega}_F}+ (\boldsymbol{u}-\boldsymbol{w}) \cdot \nabla \boldsymbol{u} \right)  = \nabla \cdot \boldsymbol\sigma_F(\boldsymbol u, p)
+ \boldsymbol f_F
 &  \textrm{in}\; \Omega_F(t)\times(0,T), \\
&\nabla \cdot \boldsymbol{u} = 0 & \textrm{in}\; \Omega_F(t)\times(0,T), \\
& \partial_{t} {\boldsymbol \eta}  =  \boldsymbol \xi &  \textrm{in}\; \hat{\Omega}_S\times(0,T), 
\\
&{\rho}_S \partial_{t} {\boldsymbol \xi}  =  {\nabla} \cdot \boldsymbol \sigma_S(\boldsymbol \eta) &  \textrm{in}\; \hat{\Omega}_S\times(0,T), 
\\
& {\boldsymbol{u}} \circ \mathcal{A}= {\boldsymbol \xi} & \textrm{on} \; \hat{\Gamma} \times (0,T), \label{coupling_noslip}\\
&   {J} \boldsymbol \sigma_F \boldsymbol{F}^{-T} \boldsymbol n_F+  {\boldsymbol \sigma_S } {\boldsymbol n}_S =0 & \textrm{on} \; \hat\Gamma \times (0,T), \label{fsi2}
\end{align}
\label{FSImoving}
\end{subequations}
complemented with boundary  conditions~\eqref{neumann},~\eqref{homostructure1} and~\eqref{homostructure2}, and initial conditions~\eqref{initial}.

To solve problem~\eqref{FSImoving}, we extend the approach presented in Algorithm~\ref{algorithm1}. In the following, we denote
$$
 \widehat{\boldsymbol{\sigma}_F \boldsymbol n_F}_{(\kappa)}^{n+\theta} =J^{n+\theta}_{(\kappa)} \boldsymbol\sigma_F(\boldsymbol{u}^{n+\theta}_{(\kappa)},p^{n+\theta}_{(\kappa)})  (\boldsymbol{F}^{n+\theta}_{\kappa})^{-T} \boldsymbol{n}_{F,(\kappa)}^{n+\theta}.
$$

 resulting in the following scheme:
\begin{algorithm}
\label{algorithm2}
Given ${\boldsymbol u}^{0}$ in $\hat{\Omega}_F$, and $ {\boldsymbol \eta}^0,  {\boldsymbol \xi}^0$ in $\hat{\Omega}_S$, 
we first need to compute ${\boldsymbol u}^{\theta_0}, p^{\theta_0}, {\boldsymbol u}^{1}, \boldsymbol \eta_F^1, \boldsymbol w^1$  and $ {\boldsymbol \eta}^{\theta_0}, {\boldsymbol \eta}^1,  {\boldsymbol \xi}^{\theta_0},  {\boldsymbol \xi}^{1}$  with a  second-order method. 
A monolithic method could be used.
Then, for all $n\geq 1$, compute the following steps:
\vskip 0.1 in
\noindent \sc{Step 1}.
\emph{
Set the initial guesses as the linearly extrapolated values:
\begin{align*}
& {\boldsymbol \eta}^{n+\theta}_{(0)}
=  \Big( 1 + \theta_n  \Big) {\boldsymbol \eta}^{n}  - \theta_n {\boldsymbol \eta}^{n-1} ,
\end{align*}
and similarly for ${\boldsymbol \xi}^{n+\theta_n}_{(0)}, {\boldsymbol u}^{n+\theta_n}_{(0)}$.
The pressure initial guess is defined as 
\begin{align*}
p^{n+\theta}_{(0)} = (1 + \tau) p^{n-1+\theta}
				-\tau  p^{n-2+\theta}
.
\end{align*}
For  $\kappa\geq 0$, compute until convergence the \textbf{Backward-Euler} partitioned problem:
\begin{empheq}[left= \ftext{2.2em}{Solid:} \;  {\empheqlbrace \quad}]{alignat=2}
&
  \eqmath[l]{A}{ \frac{{\boldsymbol \eta}^{n+\theta}_{(\kappa+1)} - {\boldsymbol \eta}^{n}}{\theta \tau} 
= 
{\boldsymbol \xi}^{n+\theta}_{(\kappa+1)}} 
& \eqmath[r]{B}{\mbox{ in }\hat\Omega_S,}
\notag
\\
&
  \eqmath[l]{A}{ \rho_S \frac{\boldsymbol{\xi}_{(\kappa+1)}^{n+\theta} - \boldsymbol{\xi}^n}{\theta \tau} = \nabla \cdot \boldsymbol{\sigma}_S(\boldsymbol{\eta}_{(\kappa+1)}^{n+\theta})}
& \eqmath[r]{B}{\mbox{ in }\hat\Omega_S,}
\notag
\\
&
  \eqmath[l]{A}{\alpha \boldsymbol{\xi}^{n+\theta}_{(\kappa+1)} + \boldsymbol\sigma_S(\boldsymbol{\eta}^{n+\theta}_{(\kappa+1)})\boldsymbol{n}_S = \alpha \boldsymbol{u}_{(\kappa)}^{n+\theta} -  \widehat{\boldsymbol{\sigma}_F \boldsymbol n_F}_{(\kappa)}^{n+\theta}}
& \eqmath[r]{B}{\mbox{ on } \hat\Gamma,}
\notag
\end{empheq}
\begin{empheq}[left= \ftext{2.2em}{Geo.:} \;  {\empheqlbrace \quad}]{alignat=2}
&
  \eqmath[l]{A}{
- \Delta {\boldsymbol \eta}^{n+\theta}_{F,(\kappa+1)} = 0 }
& \eqmath[r]{B}{\mbox{ in } \hat{\Omega}_F, }
\notag
\\
&
  \eqmath[l]{A}{ {\boldsymbol \eta}^{n+\theta}_{F,(\kappa+1)} = 0 }
& \eqmath[r]{B}{\mbox{ on }  \hat{\Gamma}^{in}_F \cup \hat{\Gamma}^{out}_F,}
\notag
\\
&
  \eqmath[l]{A}{ {\boldsymbol \eta}^{n+\theta}_{F,(\kappa+1)} = {\boldsymbol \eta}^{n+\theta}_{(\kappa+1)}}
& \eqmath[r]{B}{\mbox{ on } \hat\Gamma,}
\notag
\\
&
  \eqmath[l]{A}{ \boldsymbol w^{n+\theta}_{(\kappa+1)} = \frac{\boldsymbol{\eta}_{F,(\kappa+1)}^{n+\theta} - \boldsymbol{\eta}_F^n}{\theta \tau}}
  & \eqmath[r]{B}{\mbox{ in }\hat\Omega_F,}
  \notag
  \\
  &
    \eqmath[l]{A}{
    \Omega_{F,(\kappa+1)}^{n+\theta} = (\mathbf{I}+\boldsymbol{\eta}_{F,(\kappa+1)}^{n+\theta})\hat{\Omega}_F,
    }
    \notag
\end{empheq}
\begin{empheq}[left= \ftext{2.2em}{Fluid:}  \;  {\empheqlbrace \quad}]{alignat=2}
&
 \eqmath[l]{A}{\rho_F \frac{{\boldsymbol u}^{n+\theta}_{(\kappa+1)} - {\boldsymbol u}^{n}}{\theta \tau} 
 +\left({\boldsymbol u}^{n+\theta}_{(\kappa)} - {\boldsymbol w}^{n+\theta}_{(\kappa+1)} \right)\cdot \nabla {\boldsymbol u}^{n+\theta}_{(\kappa+1)}
 }
\notag
\\
&
 \eqmath[l]{A}{
\qquad = \nabla\cdot {\boldsymbol \sigma}_F ({\boldsymbol u}^{n+\theta}_{(\kappa+1)} , p^{n+\theta}_{(\kappa+1)} ) +\boldsymbol{f}_F(t^{n+\theta}) }
& \eqmath[r]{B}{ \mbox{ in }\Omega_{F,(\kappa+1)}^{n+\theta},}
\notag
\\
&
 \eqmath[l]{A}{\nabla \cdot {\boldsymbol u}^{n+\theta}_{(\kappa+1)} = 0}
& \eqmath[r]{B}{\mbox{ in }\Omega_{F,(\kappa+1)}^{n+\theta},}
\notag
\\
&
 \eqmath[l]{A}{\alpha \boldsymbol{u}^{n+\theta}_{(\kappa+1)} - \boldsymbol\sigma_F(\boldsymbol{u}_{(\kappa)}^{n+\theta},p_{(\kappa)}^{n+\theta}) \boldsymbol{n}_{F,(\kappa)}^{n+\theta}}
\notag
\\
 &
  \eqmath[l]{A}{
\qquad  = \alpha \boldsymbol{\xi}^{n+\theta}_{(\kappa+1)} -\boldsymbol\sigma_F(\boldsymbol{u}_{(\kappa+1)}^{n+\theta},p_{(\kappa+1)}^{n+\theta}) \boldsymbol{n}_{F,(\kappa+1)}^{n+\theta}}
&\eqmath[r]{B}{ \mbox{ on }\Gamma_{(\kappa+1)}^{n+\theta}.}
\notag
\end{empheq}
The converged solutions,
\begin{align}
{\boldsymbol \eta}^{n+\theta}_{(\kappa)} , 
{\boldsymbol \xi}^{n+\theta}_{(\kappa)},
{\boldsymbol u}^{n+\theta}_{(\kappa)} , 
p^{n+\theta}_{(\kappa)},
\boldsymbol w^{n+\theta}_{(\kappa)},
\boldsymbol{\eta}_{F,(\kappa)}^{n+\theta} 
\xrightarrow{\kappa\rightarrow\infty}
{\boldsymbol \eta}^{n+\theta} , 
{\boldsymbol \xi}^{n+\theta} , 
{\boldsymbol u}^{n+\theta} , 
p^{n+\theta},
\boldsymbol w^{n+\theta},
\boldsymbol{\eta}_{F}^{n+\theta} ,
\nonumber
\end{align}
then satisfy:
\begin{empheq}[left= \ftext{2.2em}{Solid:} \;  {\empheqlbrace \quad}]{alignat=2}
&
  \eqmath[l]{A}{ \frac{{\boldsymbol \eta}^{n+\theta} - {\boldsymbol \eta}^{n}}{\theta \tau} 
= 
{\boldsymbol \xi}^{n+\theta}} 
& \eqmath[r]{B}{\mbox{ in }\hat\Omega_S,}
\notag
\\
&
  \eqmath[l]{A}{ \rho_S \frac{\boldsymbol{\xi}^{n+\theta} - \boldsymbol{\xi}^n}{\theta \tau} = \nabla \cdot \boldsymbol{\sigma}_S(\boldsymbol{\eta}^{n+\theta})}
& \eqmath[r]{B}{\mbox{ in }\hat\Omega_S,}
\notag
\\
&
  \eqmath[l]{A}{ \boldsymbol\sigma_S(\boldsymbol{\eta}^{n+\theta})\boldsymbol{n}_S =  -  \widehat{\boldsymbol{\sigma}_F \boldsymbol n_F}^{n+\theta}}
& \eqmath[r]{B}{\mbox{ on } \hat\Gamma,}
\notag
\end{empheq}
\begin{empheq}[left= \ftext{2.2em}{Geo.:} \;  {\empheqlbrace \quad}]{alignat=2}
&
  \eqmath[l]{A}{
- \Delta {\boldsymbol \eta}^{n+\theta}_{F} = 0 }
& \eqmath[r]{B}{\mbox{ in } \hat{\Omega}_F, }
\notag
\\
&
  \eqmath[l]{A}{ {\boldsymbol \eta}^{n+\theta}_{F} = 0 }
& \eqmath[r]{B}{\mbox{ on }  \hat{\Gamma}^{in}_F \cup \hat{\Gamma}^{out}_F,}
\notag
\\
&
  \eqmath[l]{A}{ {\boldsymbol \eta}^{n+\theta}_{F} = {\boldsymbol \eta}^{n+\theta}}
& \eqmath[r]{B}{\mbox{ on } \hat\Gamma,}
\notag
\\
&
  \eqmath[l]{A}{ \boldsymbol w^{n+\theta} = \frac{\boldsymbol{\eta}_{F}^{n+\theta} - \boldsymbol{\eta}_F^n}{\theta \tau}}
  & \eqmath[r]{B}{\mbox{ in }\hat\Omega_F,}
  \notag
  \\
  &
    \eqmath[l]{A}{
    \Omega_{F}^{n+\theta} = (\mathbf{I}+\boldsymbol{\eta}_{F}^{n+\theta})\hat{\Omega}_F,
    }
    \notag
\end{empheq}
\begin{empheq}[left= \ftext{2.2em}{Fluid:}  \;  {\empheqlbrace \quad}]{alignat=2}
&
 \eqmath[l]{A}{\rho_F \frac{{\boldsymbol u}^{n+\theta} - {\boldsymbol u}^{n}}{\theta \tau} 
 +\left({\boldsymbol u}^{n+\theta} - {\boldsymbol w}^{n+\theta} \right)\cdot \nabla {\boldsymbol u}^{n+\theta}
 }
\notag
\\
&
 \eqmath[l]{A}{
\qquad = \nabla\cdot {\boldsymbol \sigma}_F ({\boldsymbol u}^{n+\theta}, p^{n+\theta} ) +\boldsymbol{f}_F(t^{n+\theta}) }
& \eqmath[r]{B}{ \mbox{ in }\Omega_{F}^{n+\theta},}
\notag
\\
&
 \eqmath[l]{A}{\nabla \cdot {\boldsymbol u}^{n+\theta}= 0}
& \eqmath[r]{B}{\mbox{ in }\Omega_{F}^{n+\theta},}
\notag
\\
&
 \eqmath[l]{A}{ \boldsymbol{u}^{n+\theta}=  \boldsymbol{\xi}^{n+\theta}}
&\eqmath[r]{B}{ \mbox{ on }\Gamma^{n+\theta}.}
\notag
\end{empheq}
}
\vskip 0.1 in
\noindent \sc{Step 2}.
\emph{
Now evaluate the following:
\begin{empheq}[left= \ftext{2.2em}{Solid:}  \;  {\empheqlbrace} \quad]{alignat=2}
&  \eqmath[l]{A}{{\boldsymbol \eta}^{n+1} 
= \frac{1}{\theta} {\boldsymbol \eta}^{n+\theta} - \frac{1-\theta}{\theta}{\boldsymbol \eta}^{n} }
&
\eqmath[r]{B}{ 
\mbox{ in }\hat\Omega_S,}
\notag
\\
&
 \eqmath[l]{A}{
{\boldsymbol \xi}^{n+1} 
= 
\frac{1}{\theta} {\boldsymbol \xi}^{n+\theta} - \frac{1-\theta}{\theta}{\boldsymbol \xi}^{n}
}
&
\eqmath[r]{B}{ 
\mbox{ in }\hat\Omega_S,}
\notag
\end{empheq}
{\color{red}
\begin{empheq}[left= \ftext{2.2em}{Geo.:} \;  {\empheqlbrace \quad}]{alignat=2}
&
  \eqmath[l]{A}{
- \Delta {\boldsymbol \eta}^{n+1}_{F} = 0 }
& \eqmath[r]{B}{\mbox{ in } \hat{\Omega}_F, }
\notag
\\
&
  \eqmath[l]{A}{ {\boldsymbol \eta}^{n+1}_{F} = 0 }
& \eqmath[r]{B}{\mbox{ on }  \hat{\Gamma}^{in}_F \cup \hat{\Gamma}^{out}_F,}
\notag
\\
&
  \eqmath[l]{A}{ {\boldsymbol \eta}^{n+1}_{F} = {\boldsymbol \eta}^{n+1}}
& \eqmath[r]{B}{\mbox{ on } \hat\Gamma,}
\notag
\\
&
  \eqmath[l]{A}{ \boldsymbol w^{n+1} = \frac{\boldsymbol{\eta}_{F}^{n+1} - \boldsymbol{\eta}_F^{n+\theta}}{(1-\theta) \tau}}
  & \eqmath[r]{B}{\mbox{ in }\hat\Omega_F,}
  \notag
  \\
  &
    \eqmath[l]{A}{
    \Omega_{F}^{n+1} = (\mathbf{I}+\boldsymbol{\eta}_{F}^{n+1})\hat{\Omega}_F,
    }
    \notag
\end{empheq}
}
\begin{empheq}[left= \ftext{2.2em}{Fluid:}  \;  {\empheqlbrace} \quad]{alignat=2}
& 
\eqmath[l]{A}{
{\boldsymbol u}^{n+1} = \frac{1}{\theta} {\boldsymbol u}^{n+\theta} - \frac{1-\theta}{\theta} {\boldsymbol u}^{n}}
&
\eqmath[r]{B}{ \mbox{ in }\Omega_F^{n+\theta},}
\notag
\end{empheq}
}
Set $n=n+1$, and go back to Step 1.
\end{algorithm}
The problems above are complimented with the same boundary conditions as in Algorithm~\ref{algorithm1}.


\fi

\section{Numerical Examples}~\label{numerics}
In this section, we investigate the accuracy and the rates of convergence of the proposed method. To discretize the problem in space, we use the finite element method with uniform, conforming meshes, and denote the mesh size by $h$. The numerical method is implemented in the finite element solver FreeFem++~\cite{hecht2012new}.
 The benchmark problem presented in Example 1 is based on the method of manufactured solutions. In this example, we compute the convergence rates obtained with 
Algorithm~\ref{algorithm1} for different values of parameters $\theta, \alpha$ and the tolerance $\epsilon$. In the second example, we consider a benchmark problem commonly used to test FSI solvers. Using this example, we compare the results obtained using the proposed method  to a monolithic scheme and a loosely-coupled partitioned scheme.

\subsection{Example 1}

In this example we use a method of manufactured solutions to investigate the accuracy of the computational method presented in Algorithm~\ref{algorithm1}.
For this purpose, the set of problems we are solving is based on the time-dependent Stokes equations and elastodynamics equations with added forcing terms:
\begin{align*}
& \rho_F \partial_t \boldsymbol{u}  = \nabla \cdot \boldsymbol\sigma_F(\boldsymbol u, p) + \boldsymbol{f}_F&  \textrm{in}\; \Omega_F \times(0,T), 
\\ 
&\nabla \cdot \boldsymbol{u} = g & \textrm{in}\; \Omega_F \times(0,T),
\\
&{\rho}_S \partial_{t} {\boldsymbol \xi}  =  {\nabla} \cdot \boldsymbol \sigma_S(\boldsymbol \eta) + \boldsymbol{f}_S&  \textrm{in}\; {\Omega}_S\times(0,T). 
\end{align*}

The FSI problem is defined in a unit square domain such that the fluid domain resides in the lower half, $\Omega_F=(0,1) \times (0, 0.5)$, and the solid domain occupies the upper half,  $\Omega_S=(0,1) \times (0.5, 1)$. We use the following physical parameters: $\lambda_S=\mu_S= \rho_S= \rho_F=\mu_F=1$.
The exact solutions are given by:
\begin{align*}
  \boldsymbol{\eta}_{ref}
   &= 
    \begin{bmatrix}
	10^{-3}2x(1-x)y(1-y)e^t \\ 10^{-3}x(1-x)y(1-y) e^t 
   \end{bmatrix},
   \\
\boldsymbol{u}_{ref}
  & =
   \begin{bmatrix}
   	10^{-3}2x(1-x)y(1-y)e^t \\ 10^{-3}x(1-x)y(1-y)e^t 
   \end{bmatrix},
   \\
   p_{ref} & =-10^{-3} e^t \lambda_S \left(2(1-2x)y(1-y) + x(1-x)(1-2y)\right). 
\end{align*}
Using the exact solutions, we compute the forcing terms $\boldsymbol f_F, g$ and $\boldsymbol f_S$. We impose  Dirichlet boundary conditions on the bottom of the fluid domain, and Neumann conditions on other external boundaries. The sub-iterative portion of the scheme, defined by equations~\eqref{eq:Dstrctr+1/2--kappa}-\eqref{eq:Dflow+1/2--kappa}, is run until the relative errors between two consecutive approximations for the fluid velocity, structure velocity and displacement are less than a given tolerance, $\epsilon$.
To discretize the problem in space, we used $\mathbb{P}_2$ elements for the fluid velocity and solid displacement and velocity, and $\mathbb{P}_1$ elements for the  pressure.
In order to compute the rates of convergence, we first define the errors for the solid displacement and velocity, and fluid velocity as:
\begin{align*}
e_{\boldsymbol \eta} = \frac{ \left\Vert \boldsymbol \eta -\boldsymbol \eta_{ref} \right\Vert^2_{S}}{\left\Vert \boldsymbol \eta_{ref} \right\Vert^2_S},
\quad
e_{\boldsymbol \xi}=\frac{\left\Vert \boldsymbol \xi - \boldsymbol \xi_{ref} \right\Vert_{L^2({\Omega}_S)}}{\left\Vert \boldsymbol \xi_{ref} \right\Vert_{L^2({\Omega}_S)}}, 
\quad
e_F=\frac{\left\Vert \boldsymbol u- \boldsymbol u_{ref} \right\Vert_{L^2(\Omega_F)}}{\left\Vert \boldsymbol u_{ref} \right\Vert_{L^2(\Omega_F)}},
\end{align*}
respectively.

Recall that the Robin-type boundary conditions at the interface include a combination parameter, $\alpha$, which places an emphasis on the coupling condition of choice (i.e. kinematic or dynamic). In particular, case $\alpha=0$ gives the dynamic coupling condition, while $\alpha=\infty$ leads to the kinematic coupling condition. As similar generalized Robin boundary conditions have been used by other authors, finding an optimal value of $\alpha$ has been previously investigated. In particular, in~\cite{gerardo2010analysis} the authors suggest to use:
\begin{gather}
\alpha_{opt} = \frac{\rho_S H_S}{\tau} + \beta H_S \tau, 
\label{alphaformula}
\end{gather}
where $H_S$ is the height of the solid domain and 
$\beta=\displaystyle\frac{E}{1-\nu^2}(4\rho_1^2 - 2(1-\nu)\rho_2^2),$
with $E$ denoting the Young's modulus, $\nu$ representing the Poisson's ratio and $\rho_1$ and $\rho_2$ signifying the mean and Gaussian curvatures of the fluid-structure interface, respectively. In our case, we compute $\beta$ as:
$$\beta=  \frac{E}{(1-\nu^2)R^2} , $$ 
where $R$ is the height of the fluid domain. 
In addition to $\alpha_{opt},$ we explore other values: $\alpha =10, 100, 500$ and $1000.$

\begin{figure}[ht]
\centering{
\includegraphics[scale=0.5]{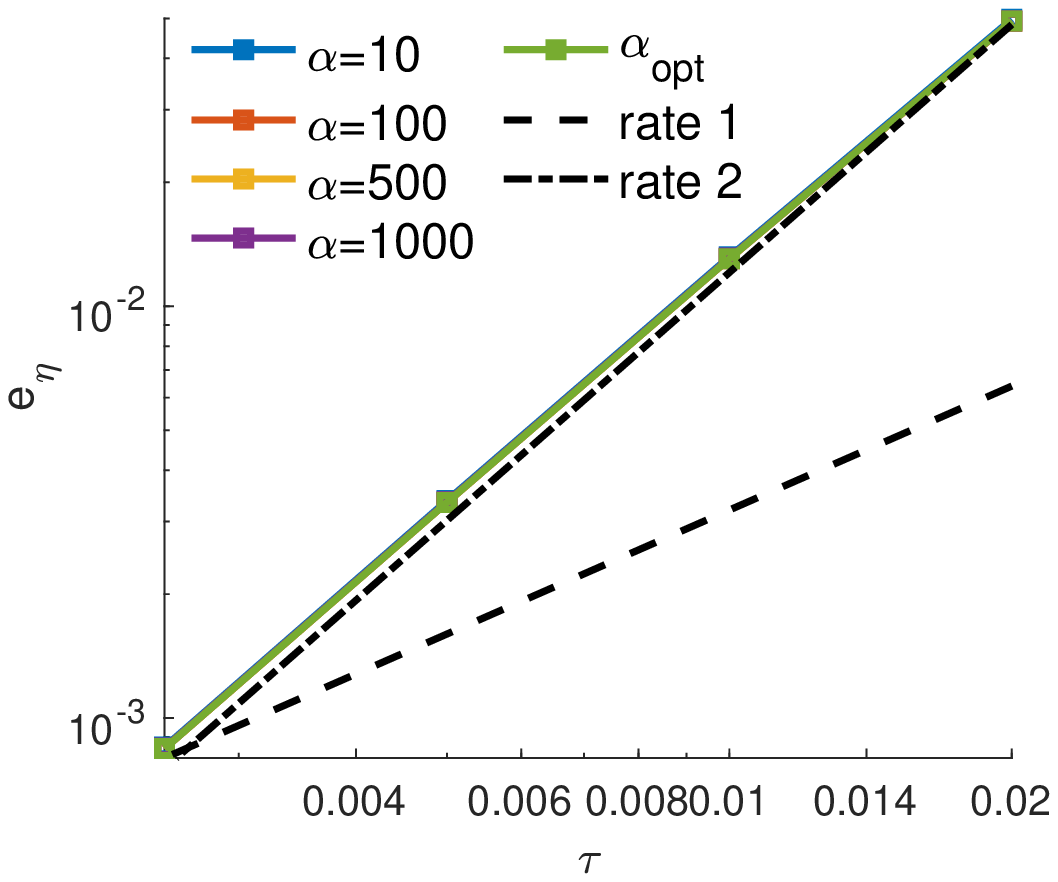}
\includegraphics[scale=0.5]{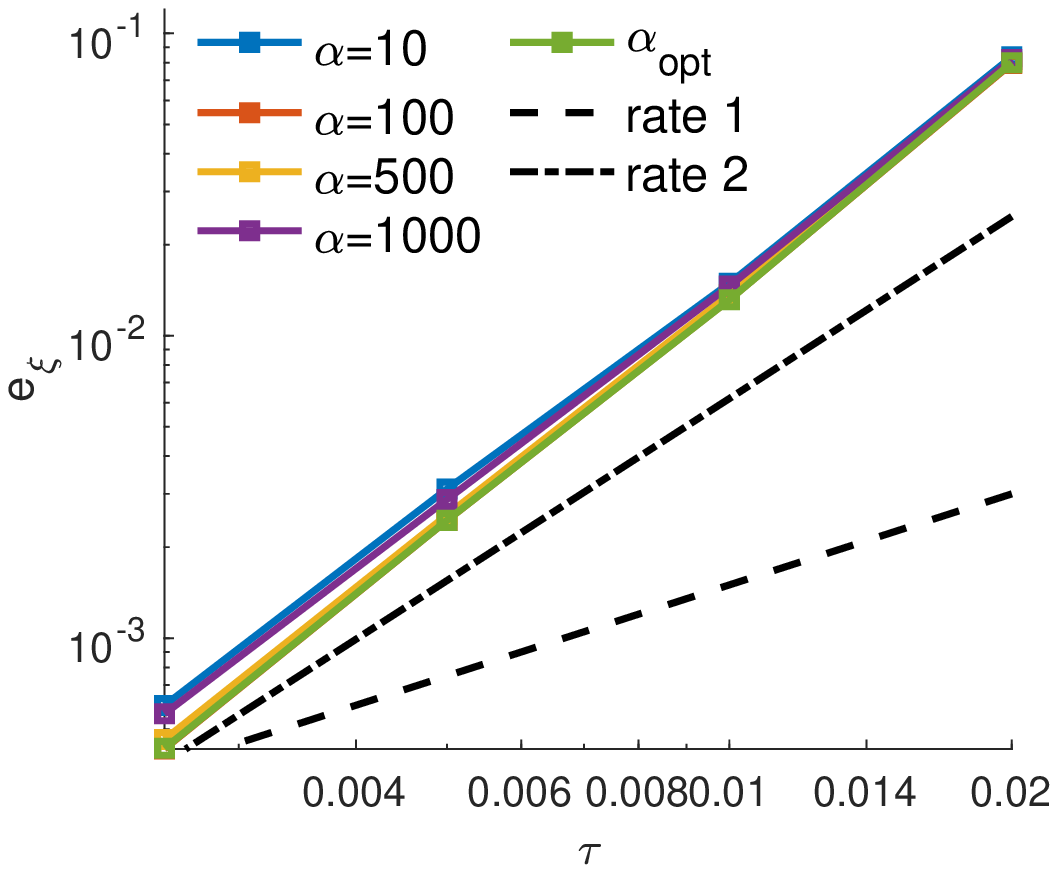}
\includegraphics[scale=0.5]{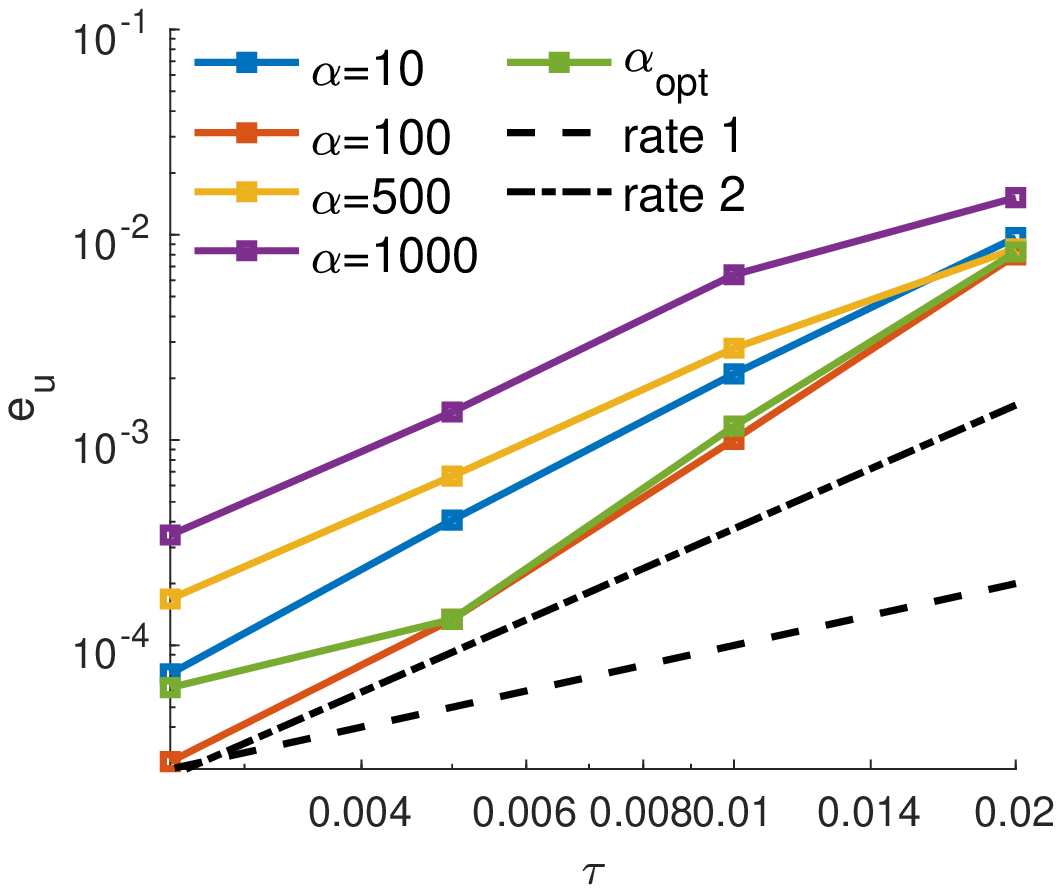}
}
\caption{Example 1: Errors obtained with $\theta=0.5$ and $\epsilon=10^{-4}$ by varying $\alpha$ for the solid displacement, $\boldsymbol \eta$, (top-left), solid velocity, $\boldsymbol \xi$, (top-right), and fluid velocity, $\boldsymbol u$, (bottom) at the final time. }
\label{BaseExample}
\end{figure}
In the first test, we set $\epsilon=10^{-4}$ and $\theta=\frac12.$ We recall that we expect to obtain the convergence rate of $\mathcal{O}(\tau^2)$ because for $\theta=\frac12,$ the discretization method corresponds to the midpoint rule.  To compute the rates of convergence, we  use the following time and space discretization parameters:
$$\left\{  \tau, h \right\} = \left\{  \frac{0.02}{2^i}, \frac{0.25}{2^i} \right\}_{i=0}^3.$$
The final time is $T=0.3$ s. The rates of convergence obtained for different values of $\alpha$ are shown in  Figure~\ref{BaseExample}.

Overall, Figure \ref{BaseExample} shows very promising convergence rates, averaging around 2 or above for all variables. The errors  for $\boldsymbol{\eta}$ are almost indistinguishable for different values of $\alpha$, showing a near-perfect convergence rate of 2. For $\boldsymbol{\xi}$, we observe that there is only a very slight disadvantage for $\alpha=10$ and 1000, with errors only slightly increased at the finest mesh size and time step; the convergence rates for all values of $\alpha$ are mostly better than 2. Finally, for $\boldsymbol{u}$, the convergence rates slightly exceed 2 for the most part with smallest errors when $\alpha=100$ and $\alpha_{opt}$.

\begin{figure}[ht]
\centering{
\includegraphics[scale=0.5]{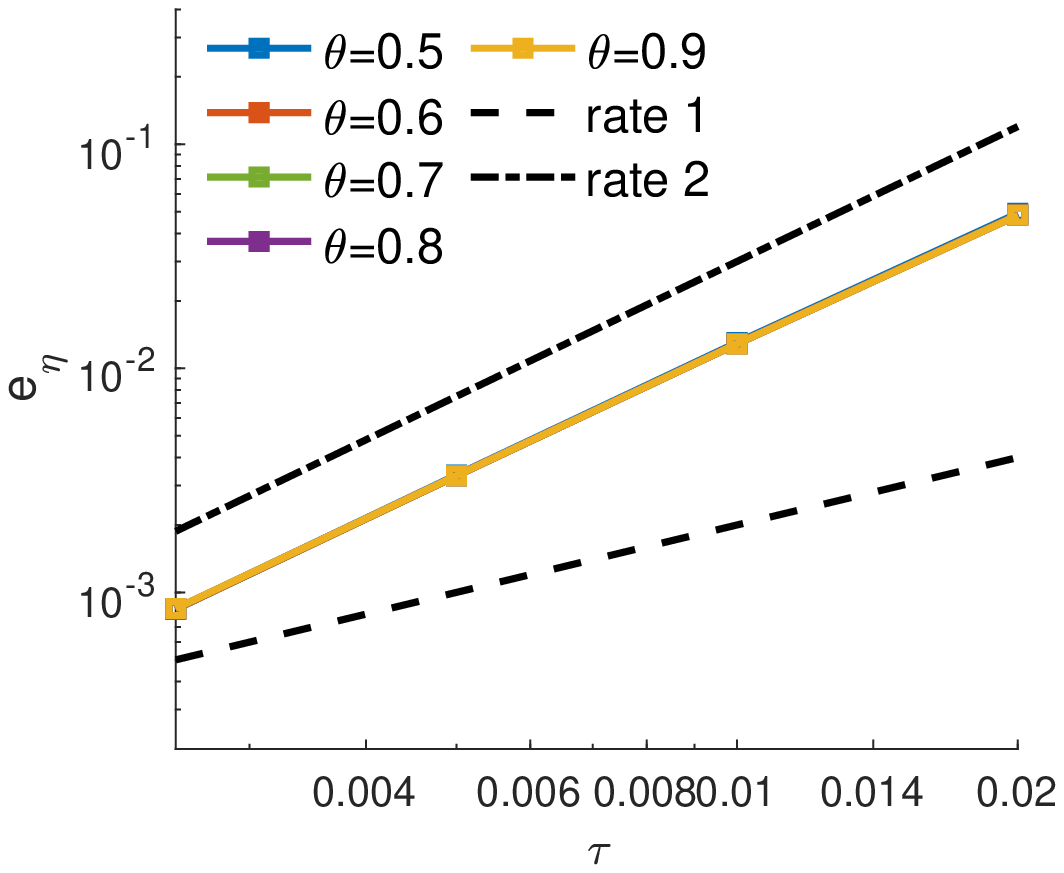}
\includegraphics[scale=0.5]{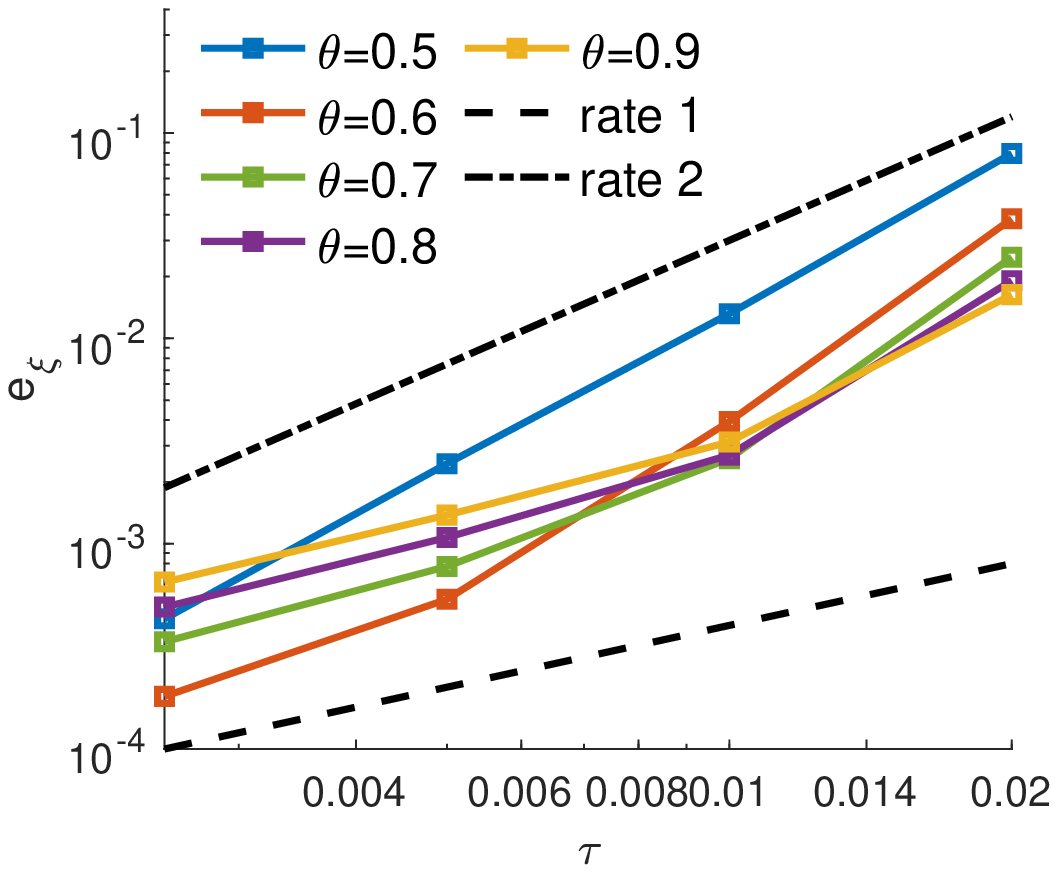}
\includegraphics[scale=0.5]{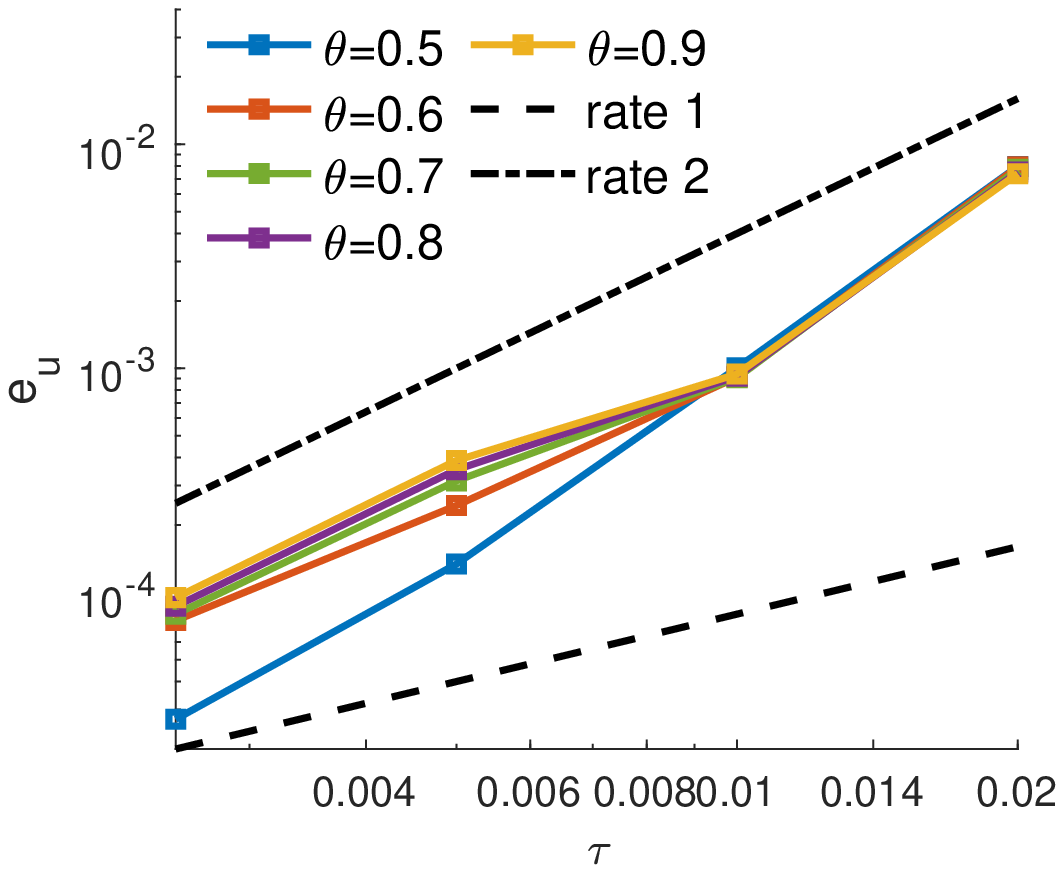}
}
\caption{Example 1: Errors obtained with $\alpha=100$ and $\epsilon=10^{-4}$ by varying $\theta$ for the solid displacement, $\boldsymbol \eta$, (top-left), solid velocity, $\boldsymbol \xi$, (top-right), and fluid velocity, $\boldsymbol u$, (bottom) at the final time. }
\label{eta_xi_u}
\end{figure}
In the next simulation, we investigate the effect of $\theta$  on the convergence rates. In particular, we use $\theta= 0.5, 0.6, 0.7, 0.8$ and $0.9$, keeping the same values of $\epsilon, \tau$ and $h$, and using $\alpha=100$. The rates of convergence are shown in Figure~\ref{eta_xi_u}. 
 As before, the errors  for $\boldsymbol{\eta}$ show a  convergence rate of 2 for all values of $\theta$. For $\boldsymbol{\xi}$, we observe that $\theta$=0.5 maintains a convergence rate of 2 or greater, while rates for other values of $\theta$ are larger than 2 for the coarsest mesh and time step and decrease to values close to 1 at the most-refined time step and mesh size. We note that $\theta$=0.5 has  larger errors than either $\theta=0.6$ or 0.7, albeit on the order of magnitude of $10^{-4}$. Finally, for $\boldsymbol{u}$, the convergence  rates are at least 2 for larger values of time step and mesh size, but similar as in the previous case, they drop down to $\mathcal{O}(\tau)$ as $\theta$ increases. The errors are the smallest when $\theta=0.5$.

In the next test,  we continue to evaluate similar conditions as in Figure~\ref{BaseExample}, but this time we use a  tolerance of $\epsilon=10^{-3}$ instead of $10^{-4}$.  As in Figure~\ref{BaseExample}, we use $\theta=\frac12$ in conjunction with a range of different values of $\alpha$. Figure~\ref{eps1e-3} shows the errors for the structure displacement (top-left), structure velocity (top-right), and fluid velocity (bottom).
  \begin{figure}[ht]
\centering{
\includegraphics[scale=0.5]{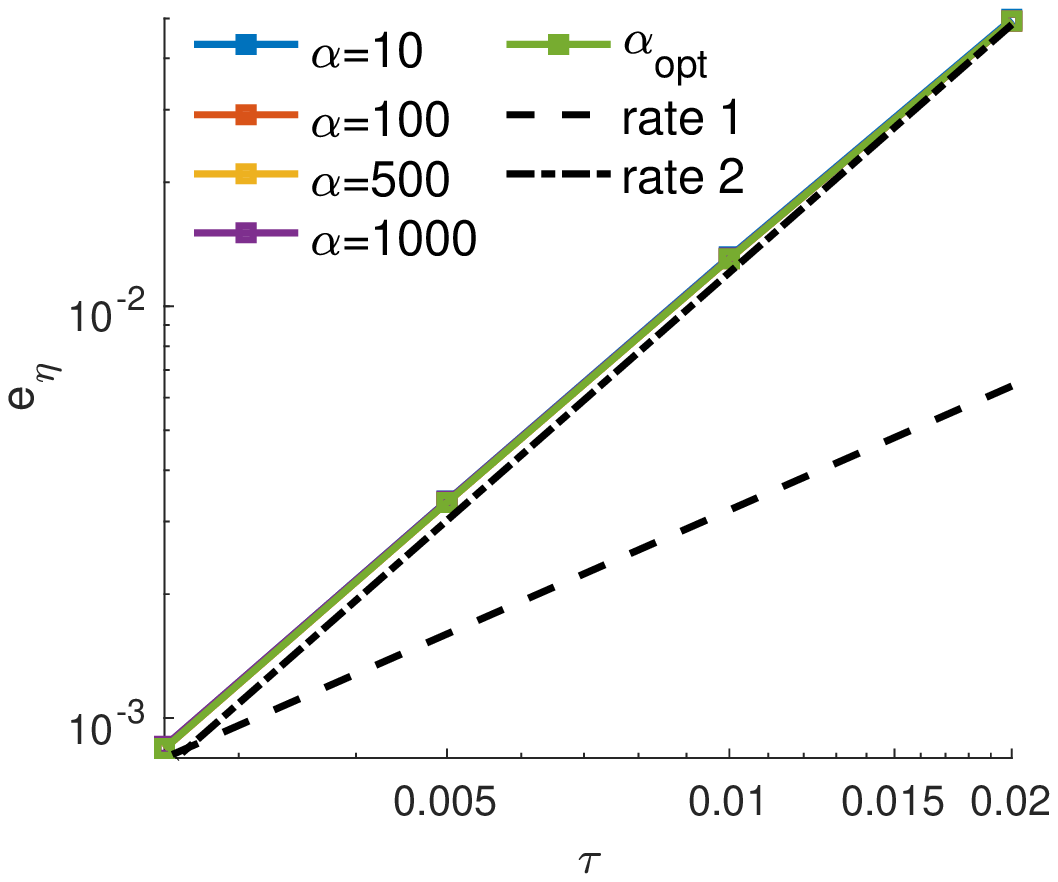}
\includegraphics[scale=0.5]{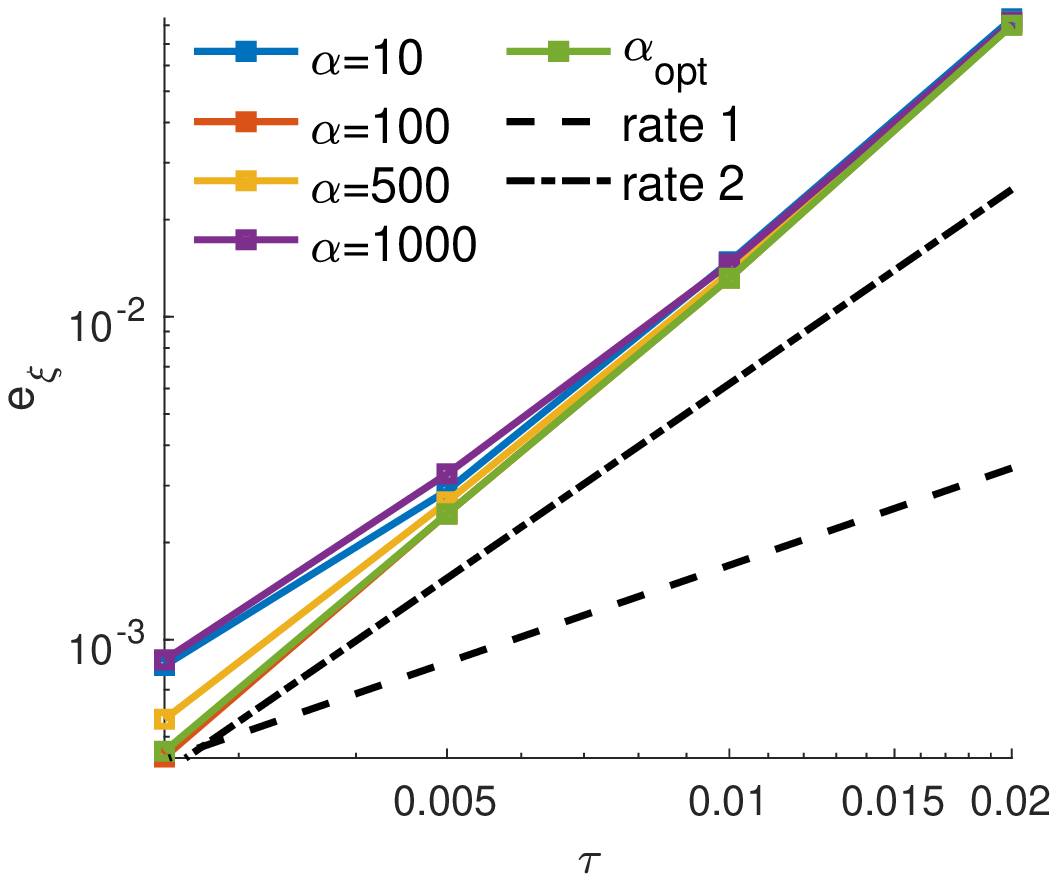}
\includegraphics[scale=0.5]{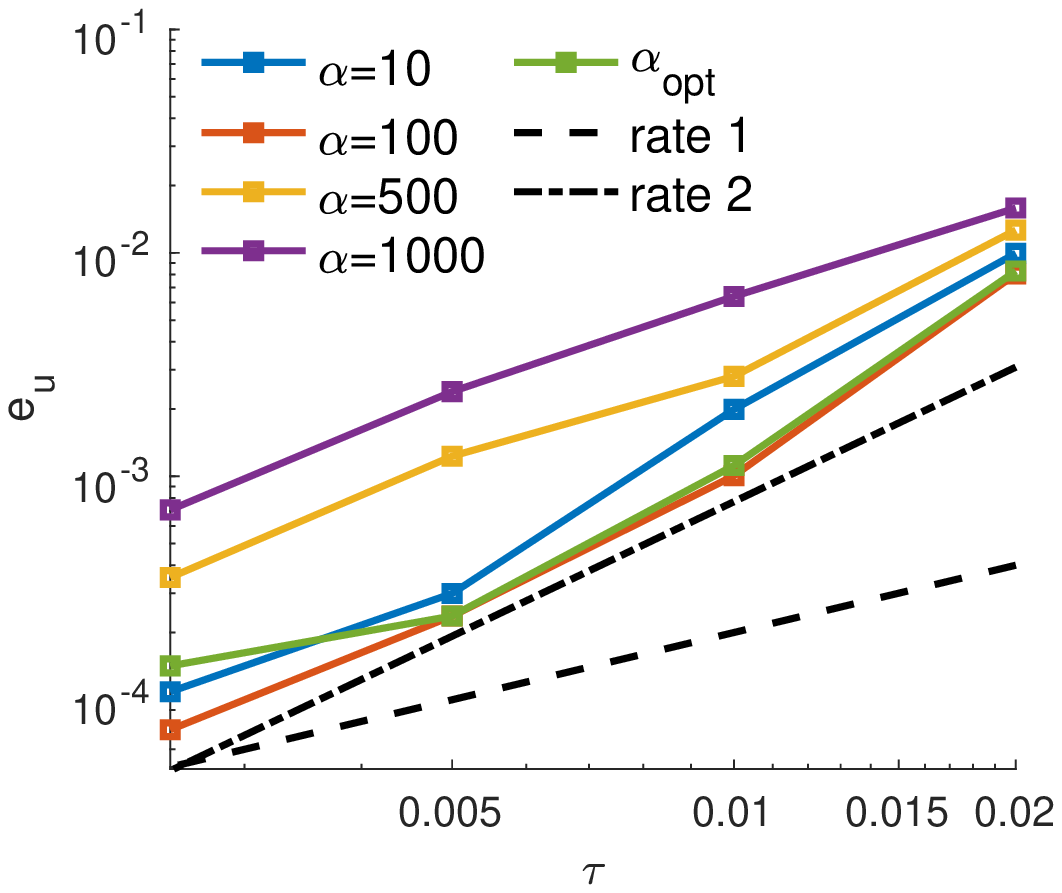}
}
\caption{Example 1: Errors obtained with $\theta=0.5$ and $\epsilon=10^{-3}$ by varying $\alpha$ for the solid displacement, $\boldsymbol \eta$, (top-left), solid velocity, $\boldsymbol \xi$, (top-right), and fluid velocity, $\boldsymbol u$, (bottom) at the final time. }
\label{eps1e-3}
\end{figure}
In this case, the rates for the solid velocity and displacement remain close to 2, while the rates for the fluid velocity become sub-optimal in most cases. 

\begin{figure}[ht]
\centering{
\includegraphics[scale=0.5]{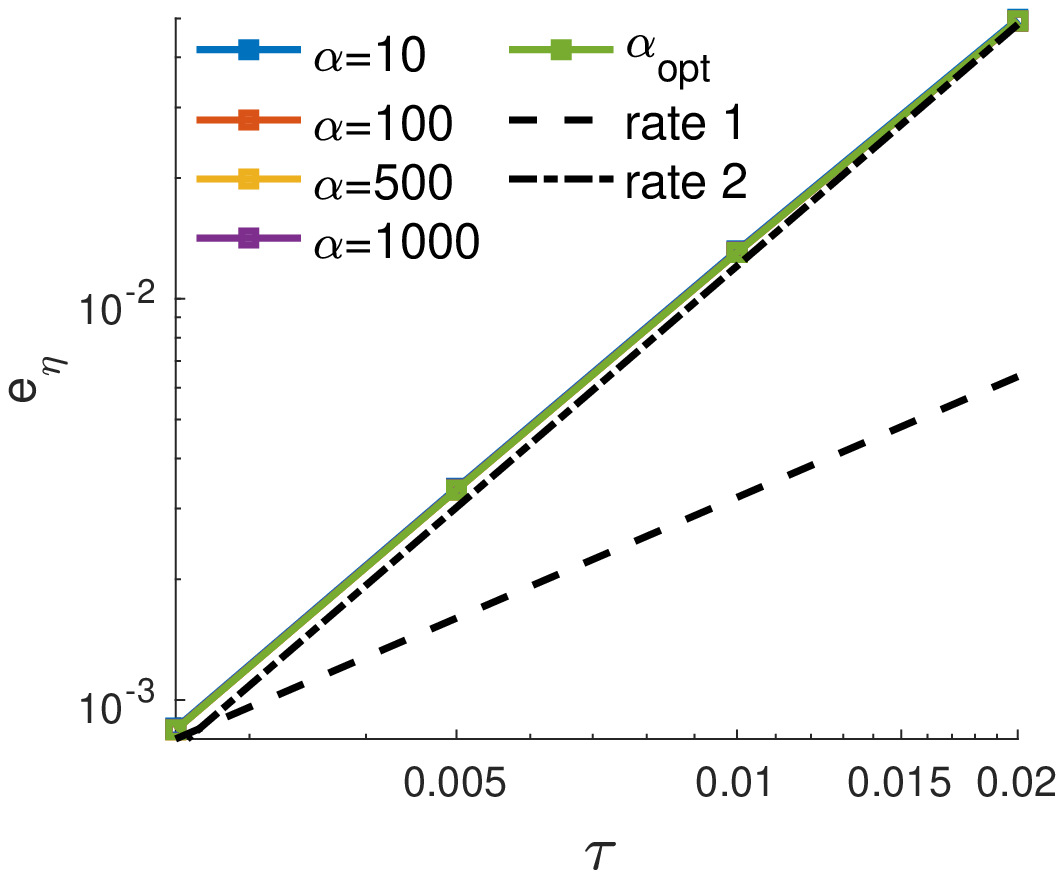}
\includegraphics[scale=0.5]{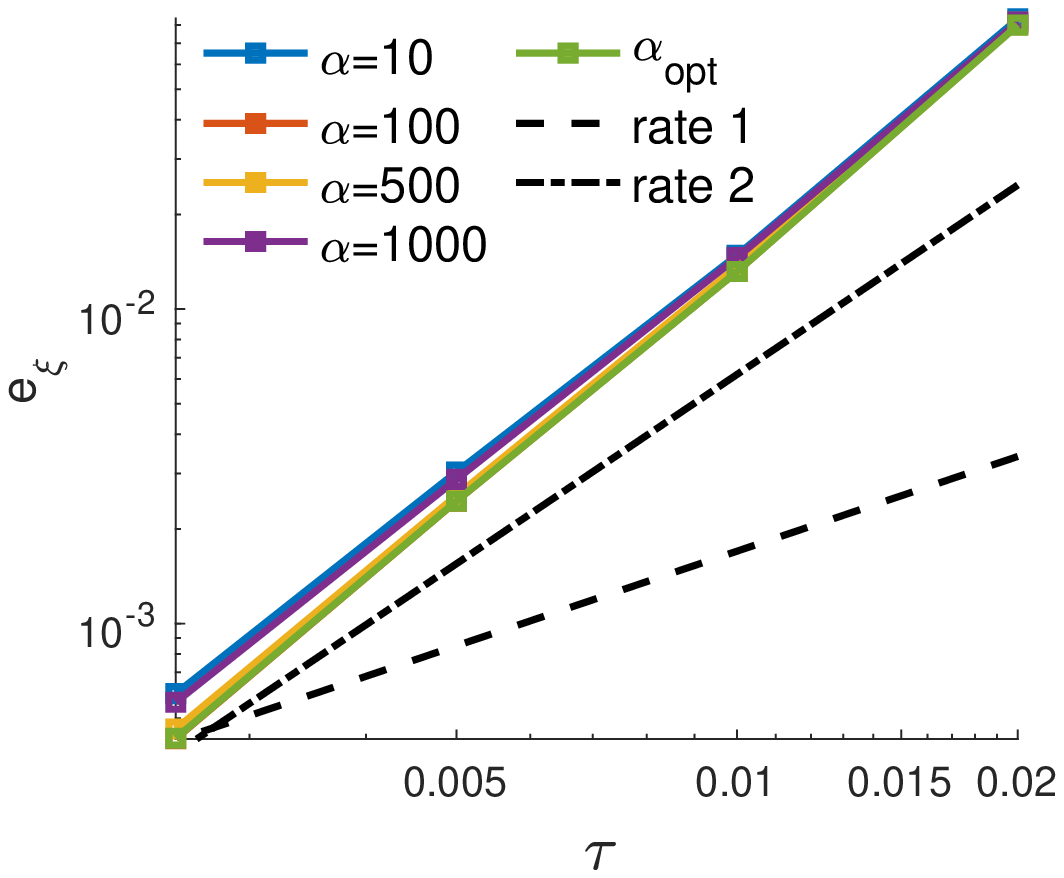}
\includegraphics[scale=0.5]{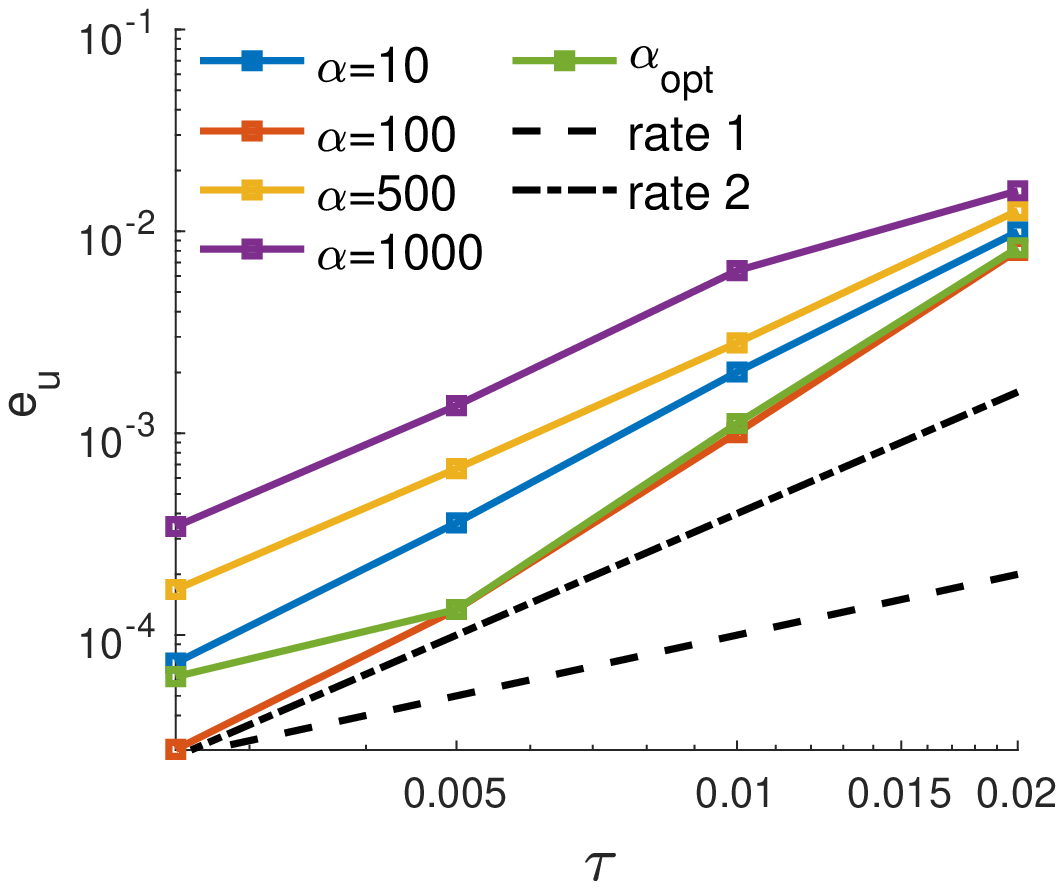}
}
\caption{Example 1: Errors, halving tolerance for each refinement, by varying $\alpha$ for the solid displacement, $\boldsymbol \eta$, (top-left), solid velocity, $\boldsymbol \xi$, (top-right), and fluid velocity, $\boldsymbol u$, (bottom) at the final time. }
\label{halvingeps}
\end{figure}
To correct the loss of accuracy that occurs for a larger value of $\epsilon$, we model the same setting as in Figure~\ref{eps1e-3}, but halve $\epsilon$ at the same rate as $\tau$, i.e., using the following set of parameters:
$$\left\{  \tau, h, \epsilon  \right\} = \left\{  \frac{0.02}{2^i}, \frac{0.25}{2^i}, \frac{10^{-3}}{2^i} \right\}_{i=0}^3.$$
Figure~\ref{halvingeps} shows the rates of convergence obtained in this case. We observe that the convergence rates for fluid and solid velocity improve when $\epsilon $ is reduced at the same rate as $\tau$. 

To further emphasize the differences in the solution for different values of $\epsilon$, shown in Figures~\ref{BaseExample},~\ref{eps1e-3} and~\ref{halvingeps}, the convergence rates for the fluid and solid velocity obtained using $\epsilon=10^{-4}$, $\epsilon=10^{-3},$ and by decreasing $\epsilon$ at the same rate as $\tau$ are shown in Table~\ref{Trates}.
\begin{table}[htb]
\begin{center}
 \begin{tabular}{c c c c c c c c c c c} 
 \hline
 $\epsilon=10^{-4}$ &$\boldsymbol{u}$&&&&&$\boldsymbol{\xi}$&&&\\ [0.5ex] 
 $\alpha: $ & $10$ & $100$ &  $500$ & $1000$& &$10$ &$100$ &$500$ &$1000$\\
 \hline
  $ \tau$  &  -&-  &- &- &  & - &-  &-  &- \\
 $ \tau/2$   &  2.21&2.99  &1.61 &1.25 &  & 2.49 &2.59  &2.56  &2.49 \\
 $ \tau/4$&2.37 &2.91  &2.07  &2.22 & & 2.26 &2.43  &2.42  &2.34 \\
 $ \tau/8$& 2.49&2.30  &1.99  &1.99 & &  2.38& 2.51 &2.46  &2.35 \\
 &  &  & & &  &  &  &  & \\
 \hline\hline
  $\epsilon=10^{-3}$ &$\boldsymbol{u}$&&&&&$\boldsymbol{\xi}$&&&\\
 $\alpha: $ & $10$ & $100$ &  $500$ & $1000$& &$10$ &$100$ &$500$ &$1000$ \\
 \hline
  $ \tau$   &  -&-  &- &- &  & - &-  &-  &-\\
 $ \tau/2$   &  2.32&3.00  &2.17 &1.31 &  & 2.50 &2.60  &2.58  &2.49 \\
 $ \tau/4$&2.74 &2.09  &1.19  &1.42 & & 2.35 &2.43  &2.36  &2.16 \\
 $ \tau/8$& 1.46&1.69  &1.80  &1.76 & &  1.80& 2.50 &2.24  &1.91 \\
 &  &  & & &  &  &  &  & \\
 \hline\hline
  $\epsilon$ changing &$\boldsymbol{u}$&&&&&$\boldsymbol{\xi}$&&& \\
 $\alpha: $ & $10$ & $100$ &  $500$ & $1000$& &$10$ &$100$ &$500$ &$1000$ \\
 \hline
  $ \tau$   &  -&-  &- &- &  & - &-  &-  &- \\
 $ \tau/2$   &  2.31&3.00  &2.17 &1.31 &  & 2.50 &2.60  &2.58  &2.49 \\
 $ \tau/4$&2.48 &2.91  &2.07  &2.22 & & 2.29 &2.43  &2.42  &2.34 \\
 $ \tau/8$& 2.31&2.30  &1.99  &1.99 & &  2.34& 2.51 &2.46  &2.35 \\
 \hline 
 \end{tabular}
\end{center}
\caption{Example 1: Rates of convergence for the fluid  and  structure velocity obtained with $\theta=\frac12$ and the following values of $\epsilon:$ $\epsilon=10^{-4}$ (top), $\epsilon=10^{-3}$ (middle), and $\epsilon$ decreasing at the same rate as $\tau$ (bottom).}
\label{Trates}
\end{table}

 \begin{figure}[ht]
\centering{
\includegraphics[scale=0.55]{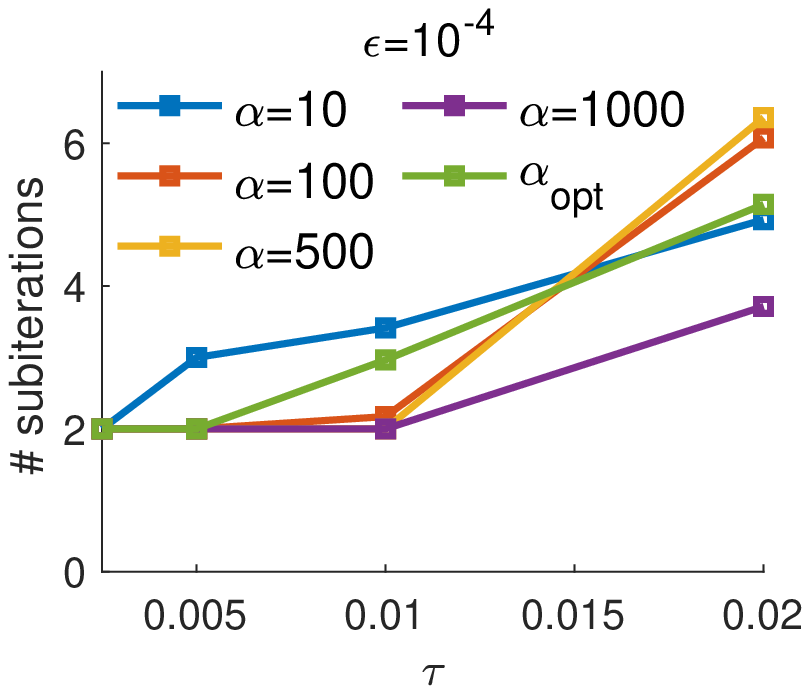}
\includegraphics[scale=0.55]{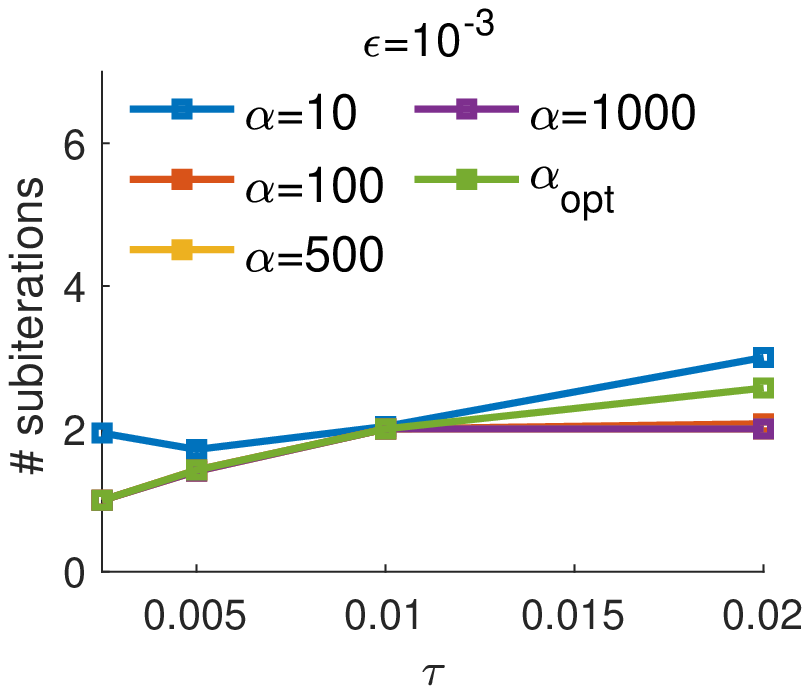}
\includegraphics[scale=0.55]{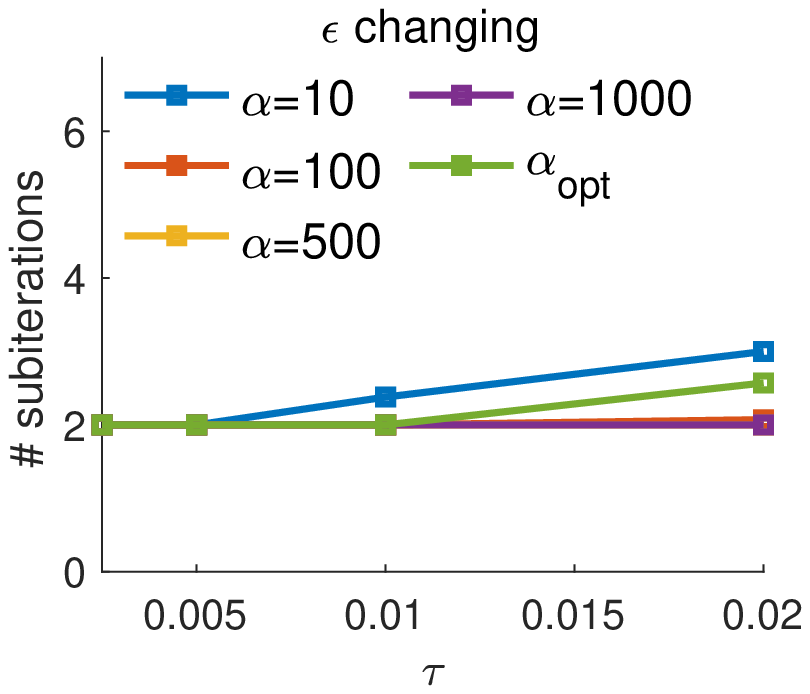}
}
\caption{Example 1: Number of sub-iterations when the tolerance is $\epsilon=10^{-4}$ (top left), $\epsilon=10^{-3}$ (top right), and when $\epsilon$ is halved in each run (bottom).}
\label{numberiterations}
\end{figure}
In  the cases presented above, we calculated the average number of sub-iterations in the sub-iterative step of our scheme. 
Figure~\ref{numberiterations} shows the average number of sub-iterations obtained with $\theta=\frac12$ and $\epsilon=10^{-4}$ (top-left), $\epsilon=10^{-3}$ (top-right), and with $\epsilon=\left\{\displaystyle\frac{10^{-3}}{2^i} \right\}_{i=0}^3$ (bottom). About at most 6 sub-iterations are needed when $\epsilon=10^{-4}$ for $\alpha=100$ and 500. However, we notice that in all cases, the number of sub-iterations decreases to about 2 when the discretization parameters decrease. As expected, a larger number of sub-iterations is required for a smaller value of $\epsilon$. However, when $\epsilon$ decreases at the same rate as $\tau$, the number of sub-iterations stays roughly the same in most cases. When looking across all three scenarios, we note that $\alpha=1000$ yields the lowest number of sub-iterations and $\alpha=10$ typically results in the highest number of sub-iterations with the exception of the coarsest $ \tau$ when the tolerance is fixed at $\epsilon=10^{-4}$.

\begin{table}[htb]
\begin{center}
 \begin{tabular}{c c c c c c c c} 
 \hline
 $ \tau$ & $h$ & $\rho_F$ & $\rho_S$ & $\epsilon$ & RN & RR & Alg. 1 \\ [0.5ex] 
 \hline\hline
$10^{-2}$ & $1.25 \cdot 10^{-1}$ & 1 & 1 & $10^{-3}$ & 70.47 & 5.43 & 2 \\ 
$5 \cdot 10^{-3}$ &  $1.25 \cdot 10^{-1}$ & 1 & 1 & $10^{-3}$ & 56.92& 3.98 & 2 \\ 
$10^{-2}$ &  $6.25 \cdot 10^{-2}$ & 1 & 1 & $10^{-3}$ & Does not converge & 10.4 & 2 \\ 
$10^{-2}$ &  $1.25 \cdot 10^{-1}$ & 1 & 10 & $10^{-3}$ & 54.07 & 11.8 & 1.03 \\ 
$10^{-2}$ &  $1.25 \cdot 10^{-1}$ & 10 & 1 & $10^{-3}$& Does not converge & 3.27 & 2 \\ 
$10^{-2}$ &  $1.25 \cdot 10^{-1}$ & 1 & 1 & $10^{-4}$ & 95.40 & 8.37 & 2.97 \\ 
 \hline
\end{tabular}
\end{center}
\caption{Example 1: The number of sub-iterations required by the Robin-Neumann (RN) method, the Robin-Robin (RR) method~\cite{badia2009robin}, and  the proposed method (Alg. 1) for different parameter values.}
 \label{tableConv}
\end{table}
Finally, we compare the number of sub-iterations required by our scheme and a couple of commonly used strongly-coupled methods for FSI problems: a Robin-Neumann scheme and a Robin-Robin scheme~\cite{badia2009robin}. We use the same parameter values in all cases, including the combination parameter $\alpha$, which is in this case set be $\alpha=\alpha_{opt}$. In the proposed method, we use $\theta=\frac12$.
 Table~\ref{tableConv} shows the number of sub-iterations required by all three methods for different parameter values. In all considered cases, the proposed method features a smaller number of sub-iterations compared to other methods. In particular, while the Robin-Neumann method did not converge when the spatial discretization parameter decreased,  in which case Robin-Robin methods required a larger number of sub-iterations, the number of sub-iterations for the proposed method did not change. We also note that for the proposed method, the number of sub-iterations decreased as the solid density increased, and increased as the tolerance, $\epsilon$, decreased. The same behavior with respect to $\epsilon$ was observed in Robin-Neumann and Robin-Robin methods.

\subsection{Example 2}
In the second example, we consider a classical benchmark problem typically used to validate FSI solvers~\cite{bukavc2014modular}. 
We consider the fluid flow in a two-dimensional channel interacting with a deformable wall.  The   fluid and structure domains are defined as ${\Omega}_F = (0,5) \times (0,0.5)$ and $\Omega_S= (0,5) \times (0.5,0.6)$, respectively. We consider the FSI problem~\eqref{flow}-\eqref{initial},  where we add a linear ``spring'' term, $\gamma \boldsymbol \eta$, to the elastodynamic equation~\eqref{solid}, yielding:
\begin{align*}
\rho_S \partial_t \boldsymbol \xi + \gamma \boldsymbol \eta = \nabla \cdot \boldsymbol \sigma_S(\boldsymbol \eta)
\qquad \textrm{in} \; {\Omega}_S \times (0,T).
\end{align*}
The term $\gamma \boldsymbol \eta$ is obtained from
the axially symmetric model and it represents a spring keeping the top and
bottom boundaries in a two-dimensional model connected~\cite{bukavc2014modular}.

We use the following parameter values: $ \rho_S=1.1$ g/cm$^3, \mu_S=1.67785 \cdot 10^6$ dyne/cm$^2$, $ \gamma=4 \cdot 10^6$  dyne/cm$^4$, $\lambda_S=8.22148  \cdot 10^7$  dyne/cm$^2$, $\rho_F$ = 1 g/cm$^3,$ and $\mu_F$ = 0.035 g/cm$\cdot $s, which are within physiologically realistic values of blood flow in compliant arteries.  In this example, we set  $\theta=\frac12, \epsilon=10^{-4}$ and $\alpha=\alpha_{opt}$, given by~\eqref{alphaformula}.  

The flow is driven by prescribing a time-dependent pressure drop at the inlet and outlet sections, as defined in~\eqref{neumann}, where
\begin{align}
    p_{in}(t)=\left\{
                \begin{array}{ll}
                  \displaystyle\frac{p_{max}}{2} \left[1-\cos \left(\displaystyle\frac{2 \pi t}{t_{max}}\right) \right],    & \text{if } t \leq t_{max}\\
                  0,    &\text{if }t > t_{max}
                \end{array}
              \right.
             ,\text{        }p_{out}=0,
\end{align}
for all $ t \in (0,T)$.
The pressure pulse is in effect for $t_{max} = 0.03$ s with maximum pressure $p_{max}=1.333 \times 10^4$ dyne/$\text{cm}^2$.  The final time is $T=12$ ms, and the time step is $\tau=10^{-4}$. At the bottom fluid boundary we prescribe symmetry conditions given by:
$$
u_y=0, \quad \frac{\partial u_x}{\partial y}=0.
$$
We assume that the structure is fixed at the edges, with zero normal stress at the external boundary, as specified in~\eqref{homostructure1}-\eqref{homostructure2}.

We use $\mathbb{P}_2- \mathbb{P}_1$ elements   for the fluid velocity and pressure, respectively, and  $\mathbb{P}_2$ elements for  the structure velocity and displacement on a mesh containing 1,000 elements in the fluid domain and 300 elements in the structure domain. 
The problem is solved using a second-order monolithic scheme, the proposed strongly-coupled scheme detailed in Algorithm~\ref{algorithm1}, and a partitioned, loosely-coupled method presented in~\cite{paper1}.

\begin{figure}[h!]
    \centering{
        \includegraphics[scale=0.43]{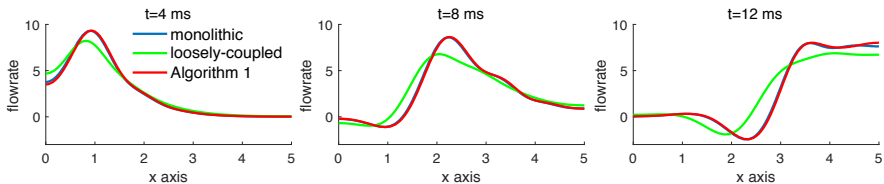}
        }
        \caption{Example 2. Fluid flowrate vs. x-axis obtained with a monolithic scheme, Algorithm~\ref{algorithm1} and a loosely coupled scheme proposed in~\cite{paper1}.}
        \label{flowrate}
\end{figure}

\begin{figure}[h!]
    \centering{
        \includegraphics[scale=0.43]{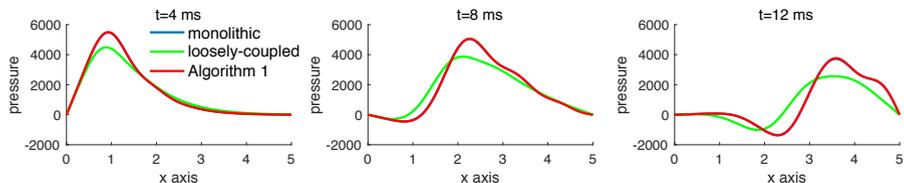}
        }
        \caption{Example 2. Fluid pressure at the centerline vs. x-axis obtained with a monolithic scheme, Algorithm~\ref{algorithm1} and a loosely coupled scheme proposed in~\cite{paper1}.}
        \label{pressure}
\end{figure}

\begin{figure}[h!]
    \centering{
        \includegraphics[scale=0.43]{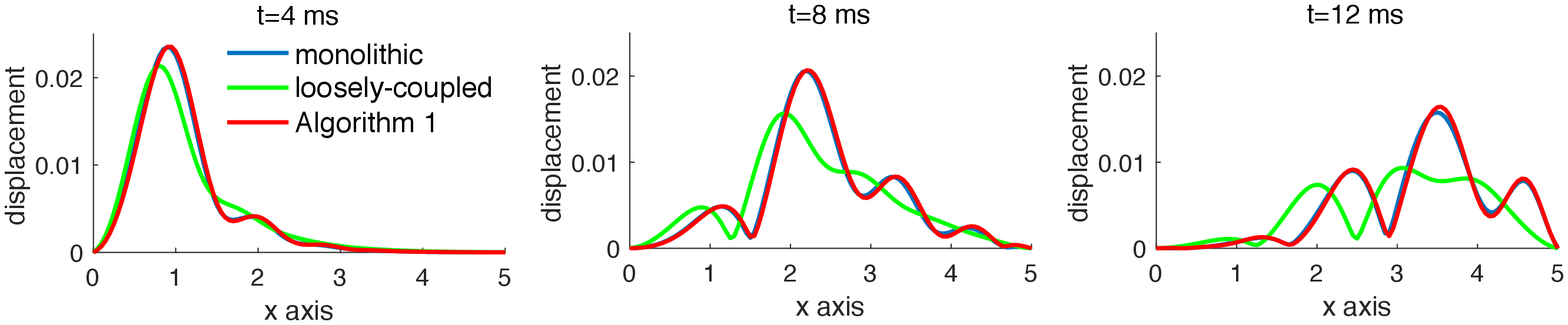}
        }
        \caption{Example 2. Fluid-structure interface displacement magnitude vs. x-axis obtained with a monolithic scheme, Algorithm~\ref{algorithm1} and a loosely coupled scheme proposed in~\cite{paper1}.}
        \label{displ_ex2}
\end{figure}

Figures~\ref{flowrate},~\ref{pressure} and~\ref{displ_ex2} show a comparison of the flowrate,  pressure at the centerline, and the interface displacement magnitude at times $t=4, 8,$ and 12 ms. An excellent agreement between the monolithic scheme and the proposed scheme is observed in all cases. We notice a larger discrepancy between the loosely-coupled scheme and the other two methods, which is due to the lower accuracy and the operator splitting error often present in partitioned methods. This is commonly corrected by taking a smaller time step in the partitioned scheme. However, the proposed method provides a greater accuracy with only few sub-iterations, combining the strengths of both monolithic and partitioned approaches.

\section{Conclusions}

In this work, we propose a novel strongly-coupled method for FSI problems with thick structures. The  method is based on the generalized Robin coupling conditions, which are split so that both the fluid and structure sub-problems are implemented using a Robin-type boundary condition at the interface.
 In order to discretize the FSI problem in time, our scheme implements a refactorization of the Cauchy's one-legged `$\theta$-like' method, where the fluid and structure sub-problems are solved in the BE-FE fashion. The BE part of the algorithm is iterated until convergence, and the FE part is equivalent to linear extrapolations, making it computationally inexpensive to solve. In this approach, the proposed method is second-order accurate when $\theta=\frac12$. Using energy estimates, we  show that the sub-iterative part of the scheme is convergent and that the method is stable provided $\theta \in [\frac12, 1].$

The theoretical expectations have been validated in numerical examples. To discretize the problem in space, we use the finite element method. We began by computing the convergence rates using the method of manufactured solutions. In order to explore the variables in our scheme, we analyzed rates amongst different values of the combination parameter, $\alpha$, the time-discretization parameter, $\theta$, and tolerance, $\epsilon$, used to measure convergence of the BE steps. We considered a wide range of values for the combination parameter, $\alpha$, as well as an optimal value $\alpha_{opt}$  proposed in~\cite{gerardo2010analysis}, which showed to be effective at maintaining optimal convergence rates and reasonably  reducing the error compared to other tested values. We obtained rates of $\mathcal{O}(\tau^2)$ when $\theta=\frac12$, while the rates decreased to orders of convergence between $\mathcal{O}(\tau)$ and $\mathcal{O}(\tau^2)$ for other values of $\theta$.  We also experienced sub-optimality in some cases if the tolerance was too large, in particular for $\epsilon=10^{-3}$. However, our results show that decreasing $\epsilon$ at the same rate as $\tau$ corrects the sub-optimalities and yields   the optimal convergence rate. 

To better understand the relation between the parameters in the problem and the computational cost of our method, we computed  the average number of sub-iterations in the BE part of the scheme. Our results show that the number of sub-iterations is reduced as the time step, $\tau$, decreases, and in  most cases considered in our study, approaches 2. We also observe that while  the case when $\epsilon=10^{-4}$ requires more sub-iterations than when $\epsilon=10^{-3}$, if we start from the latter value and decrease it at the same rate as $\tau$,  the number of sub-iterations remains roughly the same, while preserving optimal convergence rates. We also compared the number of sub-iterations required by our scheme to the ones needed by the Robin-Neumann method and the Robin-Robin method across different parameter values  and observed that in every case, our scheme has fewer sub-iterations. 

Finally, we solved an FSI problem on a benchmark example of a flow in a channel using parameters within physiologically realistic values of blood flow in compliant arteries. We compared the proposed method to both a monolithic and a non-iterative partitioned scheme, obtaining an excellent agreement with the monolithic scheme. 

A drawback of this work is that the analysis and simulations are performed assuming that the fluid-structure coupling is linear and that the fluid domain is fixed. The extensions of the method to moving domain FSI problems, as well as variable time-stepping strategies, are a focus of our on-going research.

\section{Acknowledgments}
This work was partially supported by NSF under grants DMS
  1912908 and DCSD 1934300. 

\section{Conflict of interest}
On behalf of all authors, the corresponding author states that there is no conflict of interest. 


\bibliographystyle{plain}
\bibliography{bibfile}

\end{document}